\documentclass[10pt]{amsart}
\usepackage{multicol,wrapfig,amsmath,subfigure}
\usepackage{multirow}
\usepackage{amsmath}
\usepackage{amsfonts,amssymb}
\usepackage{amsthm}
\usepackage{graphics,graphicx}
\usepackage[colorlinks,hyperindex,linkcolor=blue]{hyperref}
\hypersetup{
 pdfauthor = {A. del Pino Gomez, T. Shin},
 pdftitle= {Microflexiblity and local integrability of horizontal curves},
 pdfsubject = {Differential Geometry, tangent distributions, horizontal curves, singular curves, endpoint map}}
\usepackage{tikz-cd}
\usepackage{slashed}
\usepackage{enumitem}

\newtheorem{theorem}{Theorem}

\newtheorem{proposition}{Proposition}[section]
\newtheorem{lemma}[proposition]{Lemma}
\newtheorem{definition}[proposition]{Definition}
\newtheorem{corollary}[proposition]{Corollary}
\newtheorem{conjecture}[proposition]{Conjecture}

\newtheorem{remark}[proposition]{Remark}
\newtheorem{claim}[proposition]{Claim}

\newtheorem*{theorem*}{Theorem}

\newcommand{\R}{{\mathbb{R}}}
\newcommand{\C}{{\mathbb{C}}}
\newcommand{\D}{{\mathbb{D}}}
\newcommand{\Z}{{\mathbb{Z}}}
\newcommand{\NS}{{\mathbb{S}}}
\renewcommand{\P}{{\mathbb{P}}}

\newcommand{\ST}{{\mathcal{T}}}
\renewcommand{\SS}{{\mathcal{S}}}
\newcommand{\SR}{{\mathcal{R}}}
\newcommand{\OO}{{\mathcal{O}}}
\newcommand{\SD}{{\mathcal{D}}}
\newcommand{\SU}{{\mathcal{U}}}

\newcommand{\imm}{{\operatorname{imm}}}
\newcommand{\tang}{{\operatorname{tang}}}
\newcommand{\microreg}{{\operatorname{microreg}}}
\newcommand{\sing}{{\operatorname{inadm}}}

\newcommand{\Op}{{\mathcal{O}p}}
\newcommand{\Gr}{{\operatorname{Gr}}}
\newcommand{\Endpoint}{{\mathfrak{Ep}}}
\newcommand{\Diff}{{\operatorname{Diff}}}
\newcommand{\Ann}{{\operatorname{Ann}}}

\newcommand{\rank}{{\operatorname{rank}}}

\graphicspath{ {images/} }

\begin{document}

\title{Microflexiblity and local integrability of horizontal curves}

\date{\today}


\author{\'Alvaro del Pino}
\address{Utrecht University, Department of Mathematics, Budapestlaan 6, 3584 Utrecht, The Netherlands}
\email{a.delpinogomez@uu.nl}

\author{Tobias Shin}
\address{Stony Brook University, Department of Mathematics}
\email{tobias.shin@stonybrook.edu}

\begin{abstract}
Let $\xi$ be an analytic bracket-generating distribution. We show that the subspace of germs that are singular (in the sense of Control Theory) has infinite codimension within the space of germs of smooth curves tangent to $\xi$. We formalise this as an asymptotic statement about finite jets of tangent curves. This solves, in the analytic setting, a conjecture of Y. Eliashberg and N.M. Mishachev regarding an earlier claim by M. Gromov about the microflexibility of the tangency condition.

From these statements it follows, by an argument due to M. Gromov, that the $h$-principle holds for maps and immersions transverse to $\xi$.
\end{abstract}

\maketitle



\section{Introduction} \label{sec:introduction}

\subsection{Context of the problem} \label{ssec:context}

Let $M$ be a smooth $m$-manifold. A \textbf{distribution} $\xi$ on $M$ of rank $k$ is a section of the Grassmann bundle of $k$-planes $\Gr(TM,k)$, i.e. a smooth choice of $k$-plane $\xi_q \subset T_qM$ at each point $q \in M$. Given such $\xi$, its space of sections $\Gamma(\xi)$ is a $C^\infty$-module of vector fields. The Lie bracket inductively defines the following sequence of modules:
\[ \Gamma(\xi)^{(1)} := \Gamma(\xi), \qquad \Gamma(\xi)^{(n+1)} = [\Gamma(\xi)^{(n)},\Gamma(\xi)^{(n)}] \]
where the rightmost expression denotes taking the $C^\infty$-span of all possible brackets. We have thus a tower which we call the (fast) \textbf{Lie flag}:
\[ \Gamma(\xi)^{(1)} \subset \Gamma(\xi)^{(2)} \subset \Gamma(\xi)^{(3)} \subset ... \subset \Gamma(\xi)^{(n)} \subset \cdots \subset \Gamma(TM). \]
The pointwise rank of these modules may depend on the point $q \in M$, so they do not arise, in general, as spaces of sections of a distribution. The case of interest for us is when there is some $n_0$ such that $\Gamma(\xi)^{(n_0)} = \Gamma(TM)$; in this case we say that the distribution $\xi$ is \textbf{bracket-generating} of step $n_0$. This condition is generic and many families of distributions, like contact or Engel, are particular examples.

This notion plays a central role in many areas of Mathematics, including Contact Topology, Geometric Control Theory, Geometry of PDEs, and Subriemannian Geometry. It is natural, in all these settings, to study submanifolds of $M$ tangent to the distribution $\xi$; such a submanifold is said to be \textbf{horizontal} or \textbf{integral}.

In this paper, we focus on horizontal curves. For us, two properties are of interest:
\begin{itemize}
\item \textbf{Local integrability}: Given a vector tangent to $\xi$, there is a smooth horizontal curve tangent to it.
\item \textbf{Microflexibility}: Given a horizontal curve $\gamma: I \to M$ and a deformation through horizontal curves $\widetilde\gamma_t: I' \to M$ over a closed subset $I' \subset I$, we can find, for small times, a global horizontal deformation $\gamma_t$ of $\gamma$ agreeing with $\widetilde\gamma_t$ over $I'$.
\end{itemize}
The first property states that, locally, there are plenty of horizontal curves. This is easy to prove: any vector field tangent to $\xi$ integrates to a family of horizontal curves. The second property asserts that horizontal curves, due to the bracket-generating condition, should behave in a flexible manner, admitting many deformations. When dealing with families, one would additionally desire for these properties to hold parametrically and relative in the parameter (this will be spelled out in detail in Subsection \ref{ssec:defIntegrabilityMicroflexibility}).

In \cite[p. 84]{Gr86}, M. Gromov states:
\begin{claim}[M. Gromov] \label{ex:Gromov}
Let $(M,\xi)$ be a manifold endowed with a bracket-generating distribution. Local integrability and microflexibility holds, parametrically and relative in the parameter, for smooth curves tangent to $\xi$.

Equivalently, in $h$-principle language: The differential relation $\SR_\tang$ describing smooth curves in $M$ tangent to $\xi$ is microflexible and locally integrable.
\end{claim}
It turns out that this Claim is actually \textbf{false} as stated. R. Bryant and L. Hsu showed in \cite{BH} that there are examples of bracket-generating distributions (the simplest ones being Engel and Martinet distributions) that possess \textbf{rigid} horizontal curves, i.e. curves which cannot be deformed relative to their endpoints. Given such a rigid curve $\gamma$, we can choose a deformation $\widetilde\gamma_t$ whose domain is a neighbourhood of the endpoints (where we require it to be fixed) and a subinterval in the interior (where we require it to be non-trivial). As such, it does not admit an extension to a global deformation, contradicting microflexibility. Understanding rigidity has been a central problem in Subriemannian Geometry since then \cite{AS,ZZ,Z}.

Y. Eliashberg and N.M. Mishachev conjectured in \cite[p. 138]{EM} that Gromov's Claim \ref{ex:Gromov} should hold if we restrict to a suitable subfamily of horizontal curves. In this paper we define such a subfamily, state a modification of Claim \ref{ex:Gromov}, and prove it when $\xi$ is analytic.

\subsection{Adjusting the claim} \label{ssec:adjust}

We need some preliminary notation and definitions: Let $(M,\xi)$ be a smooth manifold equipped with a distribution. The space of smooth horizontal maps $\gamma: [a,b] \to (M,\xi)$, endowed with the $C^\infty$-topology, is denoted by $C^\infty([a,b],M,\xi)$. The subspace of maps with initial point $\gamma(a) = q \in M$ is denoted as $C^\infty_{a,q}([a,b],M,\xi)$. Both of them are Fr\'echet manifolds.

One may then ask whether the space of horizontal maps with both ends fixed is a manifold as well. These spaces can be described as the fibres of the smooth map:
\begin{definition} \label{def:endpoint}
The \textbf{endpoint map} is:
\begin{align*}
\Endpoint: C^\infty_{a,q}([a,b],M,\xi) \quad& \longrightarrow\quad (M,\xi) \\
\Endpoint(\gamma) \quad& := \quad \gamma(b).
\end{align*}
\end{definition}

Allowing us to define:
\begin{definition}
A curve $\gamma \in C_{a,q}^\infty([a,b],M,\xi)$ is \textbf{regular} if the endpoint map $\Endpoint$ is a submersion at $\gamma$. Otherwise, $\gamma$ is said to be \textbf{singular}.
\end{definition}
That is, regularity of $\gamma$ implies that the subspace of horizontal maps with endpoints $\gamma(a)$ and $\gamma(b)$ is a manifold at $\gamma$, cut out by $\Endpoint$. It follows that any small perturbation of the endpoint $\gamma(b)$ can be followed by a perturbation of $\gamma$ itself. This is an instance of microflexibility for regular curves: a deformation defined close to the endpoints can be extended to a global one. One may show that a rigid curve is necessarily singular \cite{Bryant}.

Despite this flexibility, regular curves are not the subclass of horizontal curves that we want to work with: Whereas regularity is a global property (i.e. it may be the case that $\gamma$ is regular but some subinterval is not) microflexibility is local (we must be able to extend any deformation, given over any closed subset). As such, we should focus on curves defined by a local condition.

The obvious candidate would be the class of curves all whose subintervals are regular. Despite natural, such a class is hard to describe (because it is defined at the level of germs). Instead, we work with \textbf{microregular curves} (Section \ref{sec:localIntegrability}, particularly Definitions \ref{def:singularJets} and \ref{def:microregularJets}). The rough idea is that a curve is microregular if its $r$-jets, for $r$ large enough, cannot possibly be $r$-jets of singular curves. In particular, microregular curves are regular over any subinterval. The advantage of microregularity is that it can be checked at the level of jets. Most of the paper has to do with detecting microregular jets.

We conjecture:
\begin{conjecture} \label{conj:main}
Let $(M,\xi)$ be bracket-generating. Then, local integrability and microflexibility hold, parametrically and relative in the parameter, for microregular curves tangent to $\xi$.
\end{conjecture}

One would like to claim, additionally, that the family of microregular curves is large (i.e. that only ``a few'' horizontal germs/jets are being discarded). A first result in this direction was loosely stated by L. Hsu in \cite[Corollary 7]{Hsu}, saying that regular curves are generic among horizontal curves of bracket-generating distributions. We also conjecture:
\begin{conjecture}\label{conj:main2}
Let $(M,\xi)$ be bracket-generating. Then:
\begin{itemize}
\item Germs of microregular curves are open and dense among horizontal germs.
\item Their complement has infinite codimension and contains all singular germs.
\end{itemize}
\end{conjecture}

\subsection{Statement of the main results} \label{ssec:statement}

In the present paper, we prove Conjectures \ref{conj:main} and \ref{conj:main2} when $\xi$ is analytic. Do note that the statements refer to smooth curves nonetheless:
\begin{theorem}\label{thm:main}
Let $(M,\xi)$ be bracket-generating and real analytic. Then, local integrability and microflexibility hold, parametrically and relative in the parameter, for microregular curves tangent to $\xi$.
\end{theorem}

Furthermore:
\begin{theorem}\label{thm:main2}
Let $(M,\xi)$ be bracket-generating and real analytic. Then:
\begin{itemize}
\item Germs of microregular curves are open and dense among horizontal germs.
\item Their complement has infinite codimension and contains all singular germs.
\end{itemize}
\end{theorem}
Theorem \ref{thm:main2} is, in fact, a key ingredient in the proof of Theorem \ref{thm:main}. 

The main idea behind both claims is as follows: Using tools from the theory of semianalytic and subanalytic sets we prove that $T^*M$ admits a stratification that is adapted to $\xi$ (in a precise sense described in Section \ref{sec:stratification}). This provides us with control on the cotangent lifts of singular curves and thus on the singular curves themselves and their jets/germs. 

We explain our precise setup in Section \ref{sec:overview}, including an outline of the proof in Subsection \ref{ssec:overview}. In the remainder of this Section, we provide several useful rephrasings of Theorem \ref{thm:main2}, as well some corollaries about the classification of maps \emph{transverse} to $\xi$.

\begin{remark}
The cotangent viewpoint we use to study singular curves is well-known in the Subriemannian Geometry literature, and stratification ideas have been used before to approach the \emph{Sard conjecture for the endpoint map} \cite{BFPR,BV,LLMV,LMOPV}.
\end{remark}

\begin{remark}\label{rem:genericCase}
Our approach may also apply to prove the same results for $\xi$ smooth and generic (i.e. lying in an open dense subset of the space of distributions defined by transversality with respect to a stratification in jet space). We leave this as an open question, but we point out that a related analysis of jets was carried out in \cite{CJT} to prove that a generic smooth $\xi$ with $\rank(\xi) \geq 3$ admits no abnormal minimisers.

A more difficult problem, which was suggested to us by D. Sullivan, reads: Can these methods also be applied to more general smooth distributions (for instance, generic with prescribed growth vector)?
\end{remark}

\subsection{Transverse maps} \label{ssec:transverse}

In \cite[p. 84]{Gr86}, Gromov uses Claim \ref{ex:Gromov} to argue that sheaves of maps transverse to bracket-generating distributions are flexible and thus satisfy the complete $h$-principle. This is still an open question, but some instances are known: The contact case has been proven in \cite[Section 14.2]{EM} and the Engel case in \cite{PP}. The present article settles the question for $\xi$ analytic:

\begin{theorem} \label{thm:transverse}
Let $(M, \xi)$ be an analytic, bracket-generating distribution. Let $V$ be a smooth manifold. Then both
\begin{itemize}
\item maps $f: V \to M$ with $df : TV \to TM \to TM/\xi$ surjective,
\item and immersions transverse to $\xi$,
\end{itemize}
satisfy an h–principle that is $C^0$--close, parametric, and relative (both in the parameter and the domain).
\end{theorem}
That is, the construction and classification up to homotopy of maps and immersions transverse to $\xi$ reduces to the study of their formal analogues (i.e. bundle maps $TV \to TM/\xi$ of maximal rank).

Let us briefly recall the strategy proposed by Gromov: Given a formal transverse map/immersion of $V$ into $M$, we extend it to a formal map $V \times \R \to M$ which is a formal horizontal immersion in the $\R$ direction and formally transverse in the $V$ directions. Then, one may use the microflexibility and local integrability of the horizontal immersion condition, as well as the fact that being transverse is an open condition, to invoke holonomic approximation, producing thus a transverse map $V \to M$.

As such, this strategy works when there are no singular curves (which are the obstruction to microflexibility). For instance, for contact structures, this was explained in detail in \cite[Section 14.2]{EM}. The proof applies as well to fat distributions.

In the presence of singular curves, one must replace horizontal curves by a suitable subclass (but the argument is otherwise the same). We claim that microregular curves are a nice replacement. This idea was used already in the Engel case \cite[Subsection 5.3]{PP}: there, the singular curves form a finite dimensional family and microregular curves are very easy to construct. 

In our setting, we must show that microregular immersions satisfy microflexibility and local integrability. This follows from:
\begin{corollary} \label{cor:immersions}
Let $(M,\xi)$ be analytic and bracket-generating. Let $\SR$ be an open differential relation for curves in $M$. Then:
\begin{itemize}
\item There is a weak equivalence between formal solutions of $\SR \cap \SR_\microreg$ and formal solutions of $\SR \cap \SR_\tang$.
\item $\SR \cap \SR_\microreg$ is microflexible and locally-integrable.
\end{itemize}
\end{corollary}
Indeed, we can take $\SR$ to be the relation defining immersions $\SR_\imm$. The Corollary will be proven in Subsection \ref{ssec:microregularJets}.

Corollary \ref{cor:immersions}, together with Gromov's argument, yield Theorem \ref{thm:transverse}. We invite the reader to complete the proof by referring to the contact \cite[Section 14.2]{EM} and Engel \cite[Subsection 5.3]{PP} cases.

\subsection{Microregularity is generic for regular families} \label{ssec:Thom}

According to Theorem \ref{thm:main2}, singular germs have infinite codimension within the space of all horizontal germs. We can transform this into a global genericity statement about families of microregular curves. The caveat is that, due to the phenomenon of rigidity, this cannot hold for arbitrary families of horizontal curves. Instead:
\begin{theorem}\label{thm:Thom}
Let $(M,\xi)$ be bracket-generating and real analytic. Fix an integer $a$. Then, any family of regular horizontal curves $(\gamma_k)_{k \in K}$ can be $C^a$-perturbed to yield a family of microregular horizontal curves.

This holds relative in the parameter. It also holds relative in the domain as long as the curves are regular in the complement.
\end{theorem}
This should be understood as a form of Thom transversality with respect to the locus of non-microregular jets, as long as we start with a regular family. We prove it in Section \ref{sec:Thom}.

\textbf{Acknowledgements:} The authors would like to thank F.J. Mart\'inez-Aguinaga for many insightful discussions when this project started. They are also indebted to Dennis Sullivan for helpful conversations and suggestions regarding the phrasing of some of the results. Lastly, they are thankful to Lucas Dahinden, Pablo Portilla, Fran Presas and Igor Zelenko for providing comments on a preliminary version of the paper.

\section{Setup and overview of the proof}\label{sec:overview}

For the rest of the article we fix a real analytic manifold $M$ endowed with an analytic bracket-generating distribution $\xi$. We will repeatedly make use of Gromov's notation $\Op(A)$ to denote an arbitrary open neighbourhood of $A$ of sufficiently small size.

Our goal in this Section is to introduce the conceptual setup required for Theorems \ref{thm:main} and \ref{thm:main2}. We phrase our results using the language of jet spaces, as is customary in $h$-principle. We recommend the reader to refer to the standard references \cite{EM,Gr86}.

In Subsection \ref{ssec:notationJets} we review jet spaces. Using this language, we discuss singularity and (micro)regularity at the level of jets (Subsection \ref{ssec:microregularRelation}), leading to precise rephrasings of Theorem \ref{thm:main2} in Subsection \ref{ssec:rephrasings}. In Subsection \ref{ssec:defIntegrabilityMicroflexibility} we spell out what local integrability and microflexiblity mean in the context of this paper. Lastly, in Subsection \ref{ssec:overview} we sketch the proof of Theorem \ref{thm:main2}.

\subsection{Notation: jet spaces} \label{ssec:notationJets}

Jet spaces are the central objects of this paper. To simplify our notation, we have made some slighty non-standard choices that we now explain.

\subsubsection{Jet spaces of curves}

The most important jet spaces we consider are jet spaces of curves. The relations we consider for them (being horizontal/singular/regular/characteristic) are all Diff-invariant (i.e. invariant under the action of diffeomorphisms on the domain), so the precise domain of the curve is not important. As such, we look at curves with domain $\Op(0) \subset \R$ with the origin as marked point.

We then write $J^r(M)$ for the space of $r$-jets at $0$ of curves $\Op(0) \to M$. All these spaces fit into a tower of affine bundles:
\[ \cdots \to J^{r+1}(M) \to J^r(M) \to \cdots \to J^0(M) = M. \]
This tower of jets can be extended on the left by adding a projection from $J^\infty(M)$, the space\footnote{We think of $J^\infty(M)$ as a quasi-topological space \cite[Sections 1.4 and 2.2]{Gr86}. We say that a family of maps $K \to J^\infty(M)$ is continuous if it extends to a continuous family of representative curves. Having this notion of continuity is sufficient for our $h$-principle purposes.} of infinite jets of maps $\Op(0) \to M$ based at the origin.

An important remark is that, even though we have dropped the domain $\Op(0)$ and the basepoint $0$ from the notation, these are not jets of submanifolds (for instance, the first jet is allowed to be zero). We write $j^r\gamma$ for the $r$-jet at $0$ of the curve $\gamma$. Sometimes it will be convenient to write $J^r_q(M)$ for the subspace of $J^r(M)$ consisting of jets based at $q \in M$.

\subsubsection{The reparametrisation action} \label{sssec:rho}

$\Op(0) \subset \R$ admits a (germ of) $\R^*$-action by dilations fixing the origin. This reparametrises the domain of our jets $J^r(M)$, keeping their $0$-jet fixed. As such, we obtain an action:
\[ \rho: \R^* \times J^r_q(M) \to J^r_q(M) \]
Note that $\rho$ depends on the concrete parametrisation of $\R$ we have chosen.

The map $\rho$ can be expressed explicitly in local coordinates. If we pick a chart $U \subset M \to \R^n$, any element $\sigma \in J^r_q(M)$ can be written as $(q, L_1, ..., L_r)$, where $L_i$ denotes the coefficients of the $i$th derivative (which is well-defined in such a chart). Then:
\[ \rho(a,\sigma) = (q, aL_1, a^2L_2, ..., a^rL_r). \]
The only fixed point is the jet of the constant function mapping to $q$.

Equivalently, and still in local coordinates, $J^r_q(M)$ is a graded vector space in which the pure $i$-order jets correspond to degree $i$. The action $\rho$ is precisely the homogeneous/weighted scaling. We may then take the $\rho$-quotient, yielding a weighted projectivisation $\P J^r_q(M)$.

In this paper, all operations and maps involving jets of curves are $\rho$-equivariant, and they induce maps between the corresponding (weighted) projectivisations. The fact that these are compact (unlike the original jet spaces) will play a (technical) role in Subsection \ref{ssec:singularJets} (and earlier in Corollary \ref{cor:projectCharacteristicJets}).

\subsubsection{Curves in vector bundles} \label{sssec:eta}

Our analysis of singular vs. horizontal jets is based on a microlocal criterion that lifts singular curves to characteristic curves in the cotangent bundle (Proposition \ref{prop:Hsu}, to be explained below). It is convenient to introduce some preliminary definitions with this in mind.

Let $\pi: E \to X$ be a vector bundle. There is a natural action by dilations:
\[ \eta: \R^* \times (E \setminus X) \quad\to\quad (E \setminus X). \]
The quotient of $E \setminus X$ by the $\eta$-action is the projective bundle $\P E \to X$. We often think of $\eta$ as acting trivially on the base $X$, so the projection $\pi$ is $\eta$-equivariant.

A trivial but important remark is the following:
\begin{lemma} \label{lem:equivariant}
Let $X$ be compact. Consider the projection map between jets of curves:
\[ \pi: J^r(E \setminus X) \quad\to\quad J^r(X). \] 
Lift the $\eta$-action to both spaces (trivially in the second). Then:
\begin{itemize}
\item The $\eta$ and $\rho$ actions commute with one another, yielding an $(\R^*)^2$-action on both spaces.
\item The map $\pi$ is $(\rho\oplus\eta)$-equivariant.
\item The quotients $J^r(E \setminus X)/(\rho\oplus\eta)$ and $J^r(X)/(\rho\oplus\eta)$ are smooth compact manifolds.
\item If $E \to X$ is analytic, so are the quotients and the projection map.
\end{itemize}
\end{lemma}
Again, the important property here is compactness, which we will invoke in Subsection \ref{ssec:singularJets}.

\subsubsection{Other jet spaces}

In this paper we will make use of other jet spaces. For instance, we will often pullback jets of differential forms by jets of curves. In order to keep notation light, we have opted to use the following alternate notation for these other jet spaces.

Let $E \to X$ be a smooth bundle. Then $E^{(r)}$ will denote the space of $r$-jets of sections. $E^{(r)}_{x,e}$ will be used for the subspace of those jets with basepoint $x \in X$ and value $e \in E$. Sometimes we will fix only the basepoint, writing $E^{(r)}_x$. In all cases we will use $j^r_x\sigma$ for the $r$-jet at $x$ of the section $\sigma: X \to E$.

\subsection{The main characters in this story} \label{ssec:microregularRelation}

Horizontality, singularity, and microregularity are constraints that a curve in $(M,\xi)$ may satisfy. We are interested in studying $r$-jets of curves that satisfy the differential consequences of these constraints up to order $r$. All the material in this Subsection will reappear in a more formal manner in Section \ref{sec:localIntegrability}.

\subsubsection{Horizontal jets}

Being tangent to $\xi$ defines a first order differential relation $\SR_\tang$ for curves. This relation is $\Diff$-invariant and can be understood as a subset of $J^1(M)$; we denote it by $J^1(M,\xi)$. Its differential consequences define a refined relation $J^r(M,\xi) \subset J^r(M)$ for each $r$ (including $r=\infty$) by prolongation; these are the $r$-jets of curves tangent to $\xi$ up to order $r$. We describe $J^r(M,\xi)$ in detail in Subsection \ref{ssec:EhresmannJets}.

\subsubsection{Characteristic jets}

We claim that being singular constrains the $r$-jets of an horizontal curve. This is best seen using the microlocal characterisation of singular curves introduced by L. Hsu in \cite[Theorem 6]{Hsu}:
\begin{proposition}[Hsu] \label{prop:Hsu}
Let $Z_1 \subset T^*M$ be the annihilator bundle of $\xi$. Endow it with the tautological Liouville $1$-form $\lambda$. An horizontal curve $\gamma: I \rightarrow M$ is singular if and only if there exists a lift 
\[ \widetilde{\gamma}: I \to Z_1 \setminus M \subset T^*M \qquad\textrm{s.t.}\qquad i_{\widetilde{\gamma}'}d\lambda|_{Z_1} = 0. \]
A curve satisfying this condition is said to be \textbf{characteristic}.
\end{proposition}
We will elaborate further on this symplectic formalism in Section \ref{sec:symplectic}. We then say that an element in $J^r(Z_1)$ (i.e. an $r$-jet of curve) is \textbf{characteristic} (Definition \ref{def:characteristicJets}) if it vanishes to order $r$ when evaluated in $d\lambda|_{Z_1}$ and is based away from the zero section.

Due to the nature of our argument, it is sufficient to restrict our attention to a certain subspace $\overline{J^r(Z_1)_\SS}$ of characteristic jets, called the \emph{closure of the jets of tangency type}. Roughly speaking, these are $r$-jets that can actually be tangent to characteristic curves. This is explained further in Subsection \ref{ssec:tangencyJets}.

\subsubsection{Inadmissible jets}

Motivated by Hsu' result, we project $\overline{J^r(Z_1)_\SS}$ to $J^r(M,\xi)$, yielding a closed subset $J^r(M,\xi)_\sing$ which we call the locus of \textbf{inadmissible} $r$-jets (Subsection \ref{ssec:singularJets}). We point out that being inadmissible is not exactly the same as satisfying the singularity condition up to order $r$. However, $r$-jets of singular curves are always inadmissible, so it is indeed sufficient for us to study this condition.

\subsubsection{Microregular jets}

The locus of \emph{microregular jets} $J^r(M,\xi)_\microreg$ (Definition \ref{def:microregularJets}) is the complement of $J^r(M,\xi)_\sing \subset J^r(M,\xi)$. Any germ of curve extending a microregular $r$-jet is regular.

By taking the limit as $r \to \infty$ we can define the differential relation:
\[ \SR_\microreg := J^\infty(M,\xi)_\microreg \subset J^\infty(M,\xi) \subset J^\infty(M). \]

\subsection{Rephrasings of Theorem \ref{thm:main2}} \label{ssec:rephrasings}

Using the definitions introduced in the previous Subsection, we can provide a more precise reformulation of Theorem \ref{thm:main2}:
\begin{theorem*}
Let $(M,\xi)$ be bracket-generating and real analytic. Then, $J^r(M,\xi)_\sing \subset J^r(M,\xi)$ is a closed subanalytic set whose codimension is bounded from below by $O(r)$.
\end{theorem*}

And yet another incarnation of Theorem \ref{thm:main2}:
\begin{theorem*}
Let $(M,\xi)$ be bracket-generating and real analytic. Let $K$ be a compact manifold. Then, there exists $r = O(\dim(K))$ such that any family of vectors $(v_k)_{k \in K}$ tangent to $\xi$ may be extended to a $K$-family of $r$-jets of curves $(\nu_k)_{k \in K}$ satisfying:
\begin{itemize}
\item $\nu_k$ has $v_k$ as first order data,
\item $\nu_k$ is an horizontal, microregular $r$-jet.
\end{itemize}
Furthermore, the $K$-family of extensions $(\nu_k)_{k \in K}$ may be assumed to agree with any given extension $(\widetilde{\nu_k})_{k \in \Op(\partial K)}$ along the boundary of the parameter space.
\end{theorem*}
We point out that the same statement holds if, instead of first jets, we start with a family of horizontal $a$-jets. Lastly:
\begin{theorem*}
The projection $J^\infty_\microreg(M,\xi) \to J^1(M,\xi)$ is a Serre fibration with weakly contractible fibres. In particular, there is a weak equivalence between the spaces of sections:
\[ \Gamma(J^\infty(M,\xi)_\microreg) \quad\to\quad \Gamma(J^1(M,\xi)). \]
\end{theorem*}
This statement says that the formal data associated to the microregular relation $\SR_\microreg$ is equivalent, up to homotopy, to the formal data associated to $\SR_\tang$ (which is contractible because $\xi$ is a bundle). Corollary \ref{cor:immersions} follows almost immediately from it.

All these statements will be proven in Section \ref{sec:localIntegrability}.

\subsection{What is the meaning of microflexibility and local integrability?} \label{ssec:defIntegrabilityMicroflexibility}

At this point we should define the two central concepts appearing in Theorem \ref{thm:main}. They are adaptations of the standard definitions \cite[Chapter 13]{EM} to our setting. Microflexibility in its parametric and relative forms reads:
\begin{definition} \label{def:microflexibility}
Let $K$ be a compact, finite-dimensional manifold. We say that a $K$--family of microregular horizontal curves $(\gamma_k: I \to M)_{k \in K}$ is \textbf{microflexible} if, for any:
\begin{itemize}
\item closed subset $\widetilde I \subset I$,
\item family of germs along $\widetilde I$ of microregular horizontal curves $(\widetilde{\gamma_{k,s}})_{k \in K, s \in [0,1]}$ satisfying $\widetilde{\gamma_{k,0}} = \gamma_k$,
\end{itemize}
there is a family of microregular horizontal curves $(\gamma_{k,s})_{k \in K, s \in [0,\delta]}$ extending both $(\gamma_k)$ and $(\widetilde{\gamma_{k,s}})$, for some $\delta > 0$.

Additionally, microflexibility is \textbf{relative in the parameter} if, for any:
\begin{itemize}
\item CW-complex of positive codimension $\widetilde K \subset K$,
\item family of microregular horizontal curves $(\widetilde{\widetilde{\gamma_{k,s}}})_{k \in \Op(\widetilde K), s \in [0,1]}$ extending both $(\gamma_k)$ and $(\widetilde{\gamma_{k,s}})$ in $\Op(\widetilde K)$,
\end{itemize}
the family $(\gamma_{k,s})_{k \in K, s \in [0,\delta]}$ can be chosen to extend $(\widetilde{\widetilde{\gamma_{k,s}}})_{k \in \widetilde K, s \in [0,\delta]}$ as well.
\end{definition}
In Section \ref{sec:microflexibility} we prove that parametric and relative microflexibility holds for microregular curves. I.e. any family of microregular curves is microflexible. The proof of this statement is standard, albeit technical, but it does require us to prove local integrability first.

\begin{definition} \label{def:integrability}
We say that \textbf{(parametric) local integrability} holds if the following statement is true: Let $K$ be any compact manifold. Let $(v_k)_{k \in K}$ be any finite-dimensional family of vectors tangent to $\xi$. We can then extend $(v_k)_{k \in K}$ to a family of microregular horizontal curves $(\gamma_k)_{k \in K}$ with $\gamma_k'(0) = v_k$.

We will say that local integrability is \textbf{relative in the parameter} if the following holds: Let $\tilde K \subset K$ be a CW-complex of positive codimension. Let $(\widetilde{\gamma_k})_{k \in \Op(\partial K)}$ be a family of microregular horizontal curves with $\widetilde{\gamma_k}'(0) = v_k$. Then we may choose $(\gamma_k)_{k \in K}$ so that  $\gamma_k = \widetilde{\gamma_k}$ for all $k \in \partial K$.
\end{definition}
Most of the article is dedicated to proving local integrability, which is essentially equivalent to Theorem \ref{thm:main2}.

\subsection{Some words about the proof} \label{ssec:overview}

We close this Section with a brief discussion of the techniques involved in the proof of Theorem \ref{thm:main2}: According to Hsu's result Proposition \ref{prop:Hsu}, we should study the kernel of $d\lambda|_{Z_1}$ in order to determine the space $\overline{J^r(Z_1)_\SS}$. We then project it down to obtain the space of interest $J^r(M,\xi)_\sing$.

Due to the analyticity of $\xi$, we stratify $Z_1$ into semianalytic submanifolds in a manner adapted to $d\lambda$ (Section \ref{sec:stratification}). Describing the characteristic jets tangent to a given stratum is straightforward (Subsection \ref{ssec:tangencyJets}, Proposition \ref{prop:tangencyJets}). Such a jet is said to be of \emph{tangency type}; the collection of all of them is denoted by $J^r(Z_1)_\SS$.

However, a characteristic $r$-jet may not be tangent to the stratum in which it is based. If such an $r$-jet extends to an actual characteristic curve, it will lie in the closure of the characteristic $r$-jets tangent to a (possibly different) stratum (Proposition \ref{prop:realisableJets}). This is the reason why we focus on the closure $\overline{J^r(Z_1)_\SS}$; it is a closed semianalytic subvariety of $J^r(Z_1)$.

During this process we have to pay particular attention to the locus in which the rank of $d\lambda|_{Z_1}$ is minimal. Characteristic curves contained in this locus (i.e. curves which are not \emph{Goh}) are an important subject of study in Control Theory \cite{Agr,Mon}. More generally, we have to pay attention to the vanishing of the successive curvatures of the Lie flag $\{\Gamma(\xi)^{(n)}\}_{n=1,\cdots,n_0}$. The key remark is Proposition \ref{prop:liftingHigherZ}: it states that, in the (cotangent) locus where the $(n-1)$th curvature vanishes, it is sufficient to study the characteristic jets of $\Gamma(\xi)^{(n)}$ tangent to $\xi$ (because the other jets are taken into account by the closure process described above).

Lastly, we bound from below the codimensions of $\overline{J^r(Z_1)_\SS}$ and its projection $J^r_\sing(M,\xi)$. Using the bracket-generating assumption, we prove that this bound increases linearly with $r$ (Lemma \ref{lem:dimensionCountingNR} and Proposition \ref{prop:singularJets}). This will conclude the proof. 

These arguments require some standard results from Analytic Geometry, as well as some careful Linear Algebra; we give the necessary background in Section \ref{sec:analyticGeometry}.

The outline presented in this Subsection assumes that $\xi$ is regular (i.e. that the Lie flag is a flag of constant rank distributions). To prove the result without this assumption we will need to additionally stratify $M$ according to the growth vector. Over the open strata, we will argue as we just outlined. The jets over lower-dimensional strata will be dealt with separately (Subsection \ref{ssec:submanifoldJets}).

\section{The Ehresmann lifting map} \label{sec:Ehresmann}

In this Section we work with $(W^m,\SD^l)$, a smooth manifold endowed with a (not necessarily analytic nor bracket-generating) distribution.

Our goal is discussing Definition \ref{def:EhresmannChart}: It allows us to manipulate horizontal curves by working in (semi-)local coordinates in which $\SD$ looks like a connection. We adapt this discussion to the study of horizontal jets in Subsection \ref{ssec:EhresmannJets}. Finally, in Subsection \ref{ssec:analyticEhresmann}, we reintroduce the analyticity hypothesis; this provides us with finer control regarding the structure of the space of horizontal jets.

\subsection{Ehresmann charts} \label{ssec:EhresmannChart}

We mark a preferred point $q_0 \in W$. By smoothness of $\SD$, we can choose a chart $U \subset \R^m \to W$ containing $q_0$ such that the distribution $\SD$ is \emph{graphical with respect to} $\R^l$. More precisely, identifying $U$ with its image, there is a locally defined projection $\pi: U \to \R^l$ such that 
\[ d_q\pi|_\SD: \SD_q \to T_{\pi(q)}\R^l \]
is an isomorphism for all $q$ in $U$.
\begin{definition} \label{def:EhresmannChart}
We say that $U \to W$ is an \textbf{Ehresmann chart} for $\SD$. The map $\pi: U \rightarrow \R^l$ is called the \textbf{Ehresmann projection map}.
\end{definition}
In this way, we can view our manifold $W$, locally, as (an open in) the total space of a vector bundle with base $\R^l$ and fibre $\R^{m-l}$. The distribution $\SD$ is thus viewed as an Ehresmann connection. From a Control Theory perspective: if we project with $\pi$ and then take derivatives, we are effectively passing to the \emph{space of controls}. We abuse notation and go back and forth between $W$ and $U$. We will write $\SD$ for the distribution in both cases.

\subsubsection{Ehresmann charts associated to curves}

In Section \ref{sec:Thom} we will need Ehresmann charts containing a given horizontal curve $\gamma$.
\begin{proposition} \label{prop:EhresmannChartCurve}
Let $\gamma: I \to (W,\SD)$ be a horizontal curve with domain $I$ either an interval or $\NS^1$. Then, there is:
\begin{itemize}
\item An immersion $\phi: U \subset I \times \R^{m-1} \to W$,
\item and a map $\widetilde\gamma: I \to U$,
\end{itemize}
such that:
\begin{itemize}
\item $\phi \circ \widetilde\gamma = \gamma$,
\item $\phi^*\SD$ is graphical over $I \times \R^{l-1}$.
\end{itemize}
\end{proposition}
\begin{proof}
First let us explain the proof when $\gamma$ is immersed: The bundles $TW$ and $\SD$ are trivial when pulled back to $I$. As such, we can find vector fields spanning $\SD/\langle \gamma'\rangle$ and extend them to a framing of $TW/\langle \gamma'\rangle$ around $\gamma$. We can then iteratively use their flows to build the chart around $\gamma$: The curve itself serves as the first axis, and the framing of $\SD/\langle \gamma'\rangle$ yields the next $(l-1)$-coordinates.

Now the general setting with $I$ an interval (the circle case is almost the same): We cover $I$ by a finite, ordered collection of intervals $I_i = [a_i,b_i]$, whose only non-trivial intersections are the intervals $I_i \cap I_{i+1} = [a_{i+1},b_i]$. Refining the covering, we may assume that there is an Ehresmann chart $\phi_i: U_i \subset \R^m \to W$ with image $V_i$ such that the conclusions of the Proposition apply to $\phi_i$ and the curve $\gamma|_{I_i}$. We may assume that the $U_i \subset \R^m$ are pairwise disjoint by applying translations. The idea of the proof is to piece together the $\phi_i$ to yield the claimed $\phi$. From the perspective of $W$, this piecing takes place in a neighbourhood of $\gamma([a_{i+1},b_i])$.

First: We may assume that $\gamma$ is immersed in the overlap $[a_{i+1},b_i]$. Indeed, either we can shrink $I_i$ and $I_{i+1}$ (and thus the overlap) until that is the case or $\gamma$ is constant in the overlap. In the latter case, we can shift $b_i$ to the right until we reach an immersed point. It may be the case that in doing so $b_i$ reaches $a_{i+2}$; then we delete $I_{i+1}$.

Second: We may assume that each point in $\gamma([a_{i+1},b_i])$ has a single preimage in $I_i$ (and in $I_{i+1}$). Indeed: either we can shrink the overlap to achieve it or there are other intervals in $I_i$ with the same image. In the second case, we can assume that these other intervals are immersed (by shrinking once again) and that there is a single interval $\tilde I = [c,d] \subset I_i$ with this property (there are finitely many of them by compactness and we can thus iterate the following reasoning). Then: we replace the interval $I_i$ by the three intervals $[a_i,d-\delta]$, $[d-2\delta,a_{i+1} + \delta]$, and $[a_{i+1},b_i]$; here $\delta$ is a sufficiently small constant. The proof for $I_{i+1}$ is symmetric.

Third: We shrink $V_i$ (and thus $U_i$), while still containing $\gamma|_{[a_i,a_{i+1}]}$, so that $\gamma(a_{i+1}) \in \partial V_i$. This is achieved by finding a curve connecting $\gamma(a_{i+1})$ to $\partial V_i$, and otherwise avoiding $\gamma|_{[a_i,a_{i+1}]}$, and removing a small neighbourhood. We proceed symmetrically to obtain $\gamma(b_i) \in \partial V_{i+1}$.

Fourth: We choose a curve $L_i \subset \R^m$ glueing smoothly with $\phi_i^{-1} \circ \gamma|_{[a_i,a_{i+1}]}$ at $a_{i+1}$, with $\phi_{i+1}^{-1} \circ \gamma|_{[b_i,b_{i+1}]}$ at $b_i$, and otherwise disjoint from the $\{U_j\}$. By genericity, we may assume that this curve is graphical over $\R^l$.

We now construct $\phi$. Its domain $U$ will be the union of the $\{U_j\}$, together with small neighbourhoods of the $\{L_j\}$ connecting them. We impose $\phi(L_i) = \gamma([a_{i+1},b_i])$ so we can define $\phi_{i+1/2} := \phi|_{\Op(L_i)}$ as in the first paragraph, using the fact that the latter is immersed. Recall that the construction uses a framing adapted to $\SD$ to produce the chart.

The first subtlety is that the preimage of $\gamma([a_{i+1},b_i])$ should be $L_i$, and not the first axis. However, using the fact that $L_i$ is graphical, we can reason similarly for the subsequent coordinates so that $\phi_{i+1/2}^*\SD$ is graphical over $\R^l$ as well.

The second subtlety is that $\phi_{i+1/2}$ should agree with $\phi_i$ close to $\phi_i^{-1} \circ \gamma(a_{i+1})$ and with $\phi_{i+1}$ close to $\phi_{i+1}^{-1} \circ \gamma(b_i)$. Close to $\gamma(a_{i+1})$, the coordinates provided by $\phi_i$ yield a framing. The same is true close to $\gamma(b_i)$ using $\phi_{i+1}$. The connectedness of the space of framings allow us to interpolate from one to the other as we move along $\gamma([a_{i+1},b_i])$. By taking flows this interpolating framing defines $\phi_{i+1/2}$ in $\Op(L_i)$.

Lastly, we define $\widetilde\gamma$. It is the smooth curve in $\R^m$ given by $\phi^{-1}_i \circ \gamma|_{[b_{i-1},a_{i+1}]}$ in $U_i$ and $\phi^{-1}_{i+1/2} \circ \gamma|_{[a_{i+1},b_i]}$ in $\Op(L_i)$.
\end{proof}
We will say that the map $\phi$ is an \textbf{Ehresmann chart adapted to $\gamma$}. This notion has been introduced to streamline the arguments in Section \ref{sec:Thom}, but it is not essential. Any reasoning that uses it can be replaced by covering $\gamma$ by Ehresmann charts and working chart by chart.

\subsection{The Ehresmann lifting map}

Let us recall the following standard fact: 
\begin{lemma} \label{lem:EhresmannChart}
The space $C^\infty([0,1],W,\SD)$ is a Fr\'echet manifold locally modelled on $C^\infty([0,1],\R^l) \times \R^{m-l}$. 
\end{lemma}
\begin{proof}
Let $\gamma \in C^\infty([0,1],W,\SD)$ and fix an Ehresmann chart $U$ adapted to it. We claim that the map
\begin{align*}
C^\infty([0,1],U,\SD) \quad\rightarrow\quad & C^\infty([0,1],\R^l) \times \R^{m-l}, \\
\nu \quad \to \quad & (\pi\circ \nu; \nu_{l+1}(0),\cdots,\nu_m(0)),
\end{align*}
which sends an horizontal curve to its projection and to the vertical component of its initial point $\nu(0)$, is a local homeomorphism close to $\gamma$.

Indeed, its inverse is constructed as follows: Fix $\widetilde\nu \in C^\infty([0,1],\R^l)$ contained in a neighbourhood of $\pi \circ \gamma$ and $q$ a lift of $\widetilde\nu(0)$ close to $\gamma(0)$. We use $d\pi$ to lift $\widetilde\nu'$ to a vector field defined over the slice $\pi^{-1}(\widetilde\nu)$ and tangent to $\SD$. By existence and uniqueness of ODEs, we can integrate this lifted vector field to obtain an horizontal curve $\nu$ that projects to $\widetilde\nu$ and has $q$ as basepoint. That this curve exists for all times in $[0,1]$ follows because $\widetilde\nu$ is sufficiently close to $\pi\circ\gamma$ in $\C^\infty$.

We can then transfer the Fr\'echet structure from $C^\infty([0,1],\R^l) \times \R^{m-l}$ to $C^\infty([0,1],W,\SD)$. For the purposes of this paper it is not important to look at the transition functions between charts of this form.
\end{proof}

\begin{definition} \label{def:EhresmannMap}
The \textbf{Ehresmann lifting map for curves} (associated to an Ehresmann chart) is the (partially defined) mapping
\[ E: C^\infty([0,1],\R^l) \times \R^{m-l} \rightarrow C^\infty([0,1],W,\SD) \]
defined in the previous proof.
\end{definition}
We will only use $E$ at the level of germs, so we do not need to be precise about its domain (which is an open in the vector space we wrote).

\subsubsection{Reduced endpoint maps} \label{sssec:reducedEndpoint}

Recall the endpoint map from Definition \ref{def:endpoint}. From Lemma \ref{lem:EhresmannChart} it is immediate that one can deform any horizontal curve in the direction of $\SD$. Analogously, in an Ehresmann chart $\phi: U \to (W,\SD)$, the projection to $\R^l$ of our curves can be freely manipulated.

As such, it is only in the vertical direction $\R^{m-l} \cong TW/\SD$ that the derivative of the endpoint map may fail to be surjective. For this reason we define:
\begin{definition} \label{def:reducedEndpoint}
Let $\phi: U \to (W,\SD)$ be an Ehresmann chart (potentially adapted to a curve). Then, the \textbf{reduced endpoint map} is the map:
\begin{align*}
\Endpoint_\phi: C^\infty([0,1],U,\phi^*\SD) \quad\rightarrow\quad & \R^{m-l} \\
\gamma \quad\rightarrow\quad & \pi_{m-l} \circ \gamma(1),
\end{align*}
where $\pi_{m-l}: U \to \{0\} \times \R^{m-l}$ is the projection along $\R^l$.
\end{definition}

\subsection{Variations of curves} \label{ssec:variations}

We now use the Fr\'echet model above (Lemma \ref{lem:EhresmannChart}) to describe the tangent spaces of $C^\infty([0,1],W,\SD)$. Let us recall:
\begin{definition}
A \textbf{smooth variation of an horizontal curve} $\gamma$ is a smooth map 
\[ \Gamma: [0,1] \times [0, \epsilon] \rightarrow  C^\infty([0,1],W,\SD) \qquad \textrm{s.t.}\]
\begin{itemize}
\item $\gamma_v := \Gamma|_{[0,1] \times \{v\}}$ is an horizontal curve for all $v$,
\item $\gamma_0 = \gamma$.
\end{itemize}
We say that $\frac{\partial\Gamma}{\partial v}(t,0)$ is a \textbf{variational vector field} along $\gamma$. 
\end{definition}

The tangent space of $C^\infty([0,1],W,\SD)$ at $\gamma$ is precisely the set of $\frac{\partial\Gamma}{\partial v}(t,0)$, with $\Gamma$ a variation of $\gamma$. This description of the tangent space is not explicit enough for our purposes. So instead we consider:
\begin{lemma} \label{lem:EhresmannLift}
The space of variational vector fields along $\gamma \in C^\infty([0,1],U,\SD)$ is the image of the linear mapping:
\[ d_{\pi \circ \gamma}E: C^\infty([0,1],\R^l) \oplus \R^{m-l} \rightarrow T_\gamma C^\infty([0,1],W,\SD) \subset C^\infty([0,1],\gamma^*TW). \]
Identically, in order to produce a variation of $\gamma$, we can use the following diagram:
\begin{equation*}
C^\infty([0,1],\R^l) \oplus \R^{m-l} \xrightarrow{I} C^\infty([0,1],\R^l) \times \R^{m-l} \xrightarrow{E} C^\infty([0,1],W,\SD),
\end{equation*}
where $I$ denotes integrating a variation in $\R^l$ and $E$ is the Ehresmann lifting map.
\end{lemma}
\begin{proof}
Let us prove both claims at once: Fix a variational vector field $V$ of $\pi \circ \gamma \subset \R^l$ and a vertical displacement $h \in \R^{m-l}$. The map $I$ takes $V$ and integrates it to an actual variation $\widetilde\Gamma_v := \pi(\gamma) + vV$ (this is canonical in the given Euclidean structure). Lifting the family $\widetilde\Gamma_s$ using the Ehresmann map $E$, with given basepoint $vh$, yields a variation $\Gamma_v := E \circ \widetilde\Gamma_v$ of $\gamma$. We then derive it with respect to $v$ to yield its variational vector field. This is, by construction, $d_{\pi \circ \gamma}E(V,h)$.
\end{proof}

\subsection{Jets} \label{ssec:EhresmannJets}

Our goal in this Subsection is to define the Ehresmann lifting map on the level of jets.
\begin{definition} \label{def:integralJets}
An $r$-jet in $J^r(W)$ is \textbf{horizontal} if any representative has a tangency of multiplicity $r$ with $\SD$. We let $J^r(W,\SD) \subset J^r(W)$ denote the space of horizontal $r$-jets.
\end{definition}
We remark that this definition does not depend on the actual representative chosen. Additionally, if an $r$-jet is horizontal with respect to $\SD$ at $q$, it is horizontal with respect to any other distribution having the same $r$-jet as $\SD$ at $q$.

\begin{lemma} \label{lem:integralJets}
$J^r_q(W,\SD)$ is an algebraic subvariety of $J^r_q(W)$.
\end{lemma}
\begin{proof}
Given an $(r-1)$-jet of $1$-form $j^{r-1}\alpha \in (T^*W)^{(r-1)}_q$ and an $r$-jet of curve $j^r\gamma \in J^r_q(W)$, we can pull back $j^{r-1}\alpha$ by $j^r\gamma$. This provides a map:
\[ (T^*W)^{(r-1)}_q \times J^r_q(W) \quad\to\quad (T^*\R)^{(r-1)}_0 \]
that is polynomial in both entries.

We may now choose, locally around $q$, a coframing $\{\alpha_i\}_{i=1,\cdots,m-l}$ of the annihilator of $\SD$. Using the map above, we can pair $r$-jets of curves passing through $q$ with each $j^{r-1}_q\alpha_i$, yielding an algebraic map:
\[ J^r_q(W) \quad\to\quad (T^*\R)^{(r-1)}_0 \times \cdots \times (T^*\R)^{(r-1)}_0 \]
whose zeroes are precisely the subvariety $J^r_q(W,\SD)$.
\end{proof}

If $\gamma$ is tangent to $\SD$ with $\gamma(0) = q$, its $r$-jet $j^r\gamma$ will belong to $J^r_q(W,\SD)$. The proof of the following lemma shows the converse: Any horizontal $r$-jet can be realised as the $r$-jet of an horizontal curve.
\begin{lemma} \label{lem:EhresmannLiftJets}
$J^r_q(W,\SD)$ is a smooth algebraic subvariety that can be parametrised by the algebraic map:
\[ j^r_qE: J^r_{\pi(q)}(\R^l) \to J^r_q(W,\SD) \subset J^r_q(W). \]
\end{lemma}
\begin{proof}
We use the lifting Lemma \ref{lem:EhresmannLift}. Given any $r$-jet $\sigma$ of curve in $\R^l$ we can find an actual representative germ and apply to it the map $E$. This yields a horizontal germ whose $r$-jet is, by definition, $j^r_qE(\sigma)$. This shows that $j^r_qE$ is a homeomorphism with its image.

In order to show that the map is algebraic, we now explain how $j^r_qE(\sigma)$ can be constructed inductively by solving for its coefficients (in terms of the coefficients of $\sigma$). This boils down to the usual method of solving an ODE formally.

In the coordinates provided by the Ehresmann chart $U \to \R^m$, $\SD$ is graphical over $\R^l$. We write $(x_1,\cdots,x_l)$ for the coordinates in $\R^l$ and $(y_1,\cdots,y_{m-l})$ for the coordinates in the complement $0 \times \R^{m-l}$. Graphicality tells us that $\SD$ can be defined as the kernel of $(m-l)$ $1$-forms:
\[ \alpha_i := dy_i - \sum f_j^i(x,y) dx_j. \]

Given $\gamma: I \to \R^l$ representing $\sigma$, these $1$-forms produce a system of ODEs upon restriction to $\gamma$:
\begin{align} \label{eq:ODEs}
x'(t)   \quad=\quad &\gamma'(t) \nonumber\\ 
y_i'(t) \quad=\quad & f_j^i(x(t),y(t))\gamma_j'(t),
\end{align}
whose unique solution with basepoint $q$ is the lift $E \circ \gamma$.

Such a system of ODEs can be solved coefficient by coefficient by expanding $\gamma$, $f$ and $y$ as power series. If we are given the $r$-jet of $\gamma$, the $r$-order Taylor series of $x(t)$ is given as a datum and we must solve for the coefficients of $y(t)$ up to order $r$. The second equation says that the $i$th coefficient of $y(t)$ (which is, up to a constant, the coefficient of the term of order $(i-1)$ on the left-hand side) must be some polynomial combination of coefficients in the $i$th jet of $x(t)$ and the $(i-1)$th jet of $y(t)$. As such, it can be solved for.

Therefore, the algebraic variety $J^r_q(W,\SD) \cong J^r_q(\R^m,\SD)$ is graphical over the affine subspace $J^r_{\pi(q)}(\R^l) \subset J^r_q(\R^m)$. In particular, being a graph, it is smooth.
\end{proof}

\begin{corollary} \label{cor:dimensionIntegralJets}
$J^r_q(W,\SD)$ is a smooth algebraic variety of dimension $lr$, where $l$ is the rank of $\SD$.
\end{corollary}
\begin{proof}
As seen in the proof of Lemma \ref{lem:EhresmannLiftJets}, $J^r_q(\R^m,\SD)$ can be parametrised by $J^r_{\pi(q)}(\R^l)$. The latter has the claimed dimension.
\end{proof}

\begin{definition} 
The map $j^r_qE$ defined in Lemma \ref{lem:EhresmannLiftJets} is the \textbf{Ehresmann lifting map for jets} at the point $q \in W$.
\end{definition}
We note that the left-inverse of $j^r_pE$ is the obvious (local) projection
\[ j^r_q\pi: J^r_q(W) \to J^r_{\pi(q)}(\R^l), \]
which is also algebraic.

Recall Subsections \ref{sssec:rho} and \ref{sssec:eta}:
\begin{lemma} \label{lem:reparametrisationJets}
The Ehresmann map $j^r_qE$ is $\rho$-equivariant.

In particular, it defines a smooth algebraic map between weighted projective spaces:
\[ \P j^r_qE: \P J^r_{\pi(q)}(\R^l) \to \P J^r_q(W), \]
that is a homeomorphism with its image $\P J^r_q(W,\SD)$.
\end{lemma}
\begin{proof}
For an $r$-jet of curve $\sigma$, choose a representative germ $\gamma$. The curve $\rho(a,\gamma)(t) := \gamma(at)$ is a representative of the reparameterisation $\rho(a,\sigma)$. Since the Ehresmann map is given by solving an ODE, the image of $\rho(a,\gamma)$ under the Ehresmann map is $E(\rho(a,\gamma)) = \rho(a,E(\gamma))$. The claim follows by taking jets.
\end{proof}

\subsection{The parametric Ehresmann map on jets} \label{ssec:analyticEhresmann}

If $\SD$ is smooth, the Ehresmann maps $j^r_qE$ vary smoothly with $q \in U$. As such, the same is true for the varieties $J^r_q(W,\SD)$. When $\SD$ is analytic, we have some additional structure that we will exploit later:
\begin{proposition} \label{prop:EhresmannAnalytic}
Assume $(W,\SD)$ is analytic. Then, the subspace of horizontal jets $J^r(W,\SD) \subset J^r(W)$ is a smooth analytic subvariety.

It may be (locally) parametrised as follows: Fix an analytic Ehresmann chart $U \subset W$. Then, the Ehresmann maps $j^r_qE$, for $q$ varying in $U$, can be assembled to yield an analytic embedding which is a diffeomorphism with its image:
\[ j^rE: \textrm{domain}(j^rE) \subset J^r(\R^l) \times \R^{m-l} \to J^r(W,\SD). \]
\end{proposition}
\begin{proof}
We revisit the proofs of Lemmas \ref{lem:EhresmannLiftJets} and \ref{lem:integralJets}. If $\SD$ is analytic, then we can choose an analytic coframing 
\[ \{\alpha_i := \sum_j g_j^i(x,y)dy_j + \sum_j f_j^i(x,y) dx_j\}_{i=1,\cdots,m-l} \]
for $\SD$. Contraction with the $(r-1)$-jets $j^{r-1}\alpha_i$ produces the defining equations of the subset $J^r(W,\SD)$. The analyticity of $\alpha_i$ implies the analyticity of its $(r-1)$-jet; the first claim follows.

Due to the graphicality of $\SD$ over $\R^l$, after a linear change of coordinates, we may assume that $g_i^i(p) \neq 0$ for all $i$. That implies that the reciprocals $1/g_i^i(x,y)$ are analytic functions defined in a possibly smaller chart. Dividing the $\alpha_i$ by $g_i^i(x,y)$ and then doing suitable combinations allows us to assume that $\SD$ is given by an analytic coframing of the form
\[ \{\widetilde{\alpha_i} := dy_i - \sum_j \widetilde{f}_j^i(x,y) dx_j\}_{i=1,\cdots,m-l}. \]

We now solve for the coefficients of $y(t)$ in Equation \ref{eq:ODEs}. Now the zeroeth order coefficient of $x(t)$ varies, as a point in $\R^l$, and so does the zero order coefficient of $y(t)$, as an element in $\R^{m-l}$. Since the functions $\widetilde{f}_j^i$ are analytic, the claim follows.
\end{proof}

\begin{corollary} \label{cor:dimensionIntegralJets2}
Assume $(W,\SD)$ is analytic. Then $J^r(W,\SD)$ is a smooth analytic variety of dimension $lr + n$, where $l$ is the rank of $\SD$.
\end{corollary}
\begin{proof}
We can cover $W$ by analytic Ehresmann charts as in the proof of the Proposition. In every such chart, $J^r(W,\SD)$ is an analytic graph over (a subset of) $J^r(\R^l) \times \R^{m-l}$, showing that it is a smooth analytic subvariety of $J^r(W)$. The dimension counting follows from the argument in Corollary \ref{cor:dimensionIntegralJets}. The extra $n$ dimensions arise due to the additional choice of basepoint.
\end{proof}

\section{The symplectic geometry of the annihilator} \label{sec:symplectic}

Now we review Hsu's result Proposition \ref{prop:Hsu}. With this aim in mind, we go over the symplectic/microlocal formalism in the annihilator subbundle 
\[ \pi: Z_1 := \Ann(\xi) \subset T^*M \quad\longrightarrow\quad M. \]
We do not just think of it as a subbundle of the cotangent bundle, but as a submanifold of the exact symplectic manifold $(T^*M,\lambda)$. In certain situations, $Z_1 \setminus M$ will be symplectic as well (for instance, if $\xi$ is contact, $Z_1 \setminus M$ is its symplectisation), but this is most often not the case. Hsu's theorem tells us that this failure is equivalent to the existence of singular curves.

In this Section we assume that $(M,\xi)$ is:
\begin{itemize}
\item regular (to be explained in Subsection \ref{ssec:regularity}),
\item not necessarily bracket-generating,
\item not necessarily analytic.
\end{itemize}
Once we reintroduce the analiticity assumption, we will see that $\xi$ is regular in an open dense set of $M$. This is why it is sufficient for us to restrict to the regular case now (this is justified in detail in Sections \ref{sec:stratification} and \ref{sec:localIntegrability}).

\subsection{Regular distributions} \label{ssec:regularity}

Regularity means that each $\Gamma(\xi)^{(n)}$ is the module of vector fields tangent to some distribution $\xi_n$. We then have a flag of distributions:
\[ \xi = \xi_1 \subset \xi_2 \subset \xi_3 \subset \cdots \subset \xi_{n_0}, \]
and a dual flag
\[ Z_1 \supset Z_2 \supset Z_3 \supset ... \supset Z_{n_0}, \]
consisting of the annihilators 
\[ Z_n := \Ann(\xi_n) = \{\alpha \in T^*M \,|\, \alpha(v) = 0 \textrm{ for every } v \in \xi_n \}. \]

\begin{definition}
The \textbf{curvature} of $\xi_n$ is the skew-symmetric, bilinear map
\begin{align*}
\xi_n \times \xi_n  \quad\longrightarrow\quad& \xi_{n+1}/\xi_n \subset TM/\xi_n \\
X, Y              \quad\longrightarrow\quad& [\widetilde X, \widetilde Y] \quad\textrm{mod}\quad \xi_n ,
\end{align*}
where $X$ and $Y$ are vectors in $\xi_n$ based at the same point, locally extended to tangent vector fields $\widetilde X$ and $\widetilde Y$.
\end{definition}
We leave as an exercise to show that this is well-defined using the Leibniz rule. The curvatures measure the non-integrability of $\xi$. One may dualise this notion, yielding a linear map
\begin{align*}
Z_n = \Ann(\xi_n)			\quad\longrightarrow\quad& \wedge^2 \xi_n^* \\
\alpha	\quad\longrightarrow\quad& d\alpha|_{\xi_n}.
\end{align*}

\begin{lemma} \label{lem:characterisationZ}
The annihilator $Z_{n+1}$ is spanned by those covectors $\alpha \in Z_n$ in the kernel of the dual curvature of $\xi_n$.
\end{lemma}
\begin{proof}
Let $X$ and $Y$ be local vector fields tangent to $\xi_n$. This tangency condition implies $d\alpha(X,Y) = -\alpha([X,Y])$. Since brackets of the form $[X,Y]$ span $\xi_{n+1}$, the claim follows.
\end{proof}

\subsection{The characteristic distribution is a partial connection}

We are interested in studying:
\begin{definition}
We call $\ker(d\lambda|_{Z_1})$ the \textbf{characteristic distribution} (even though it does not have constant rank).

A vector tangent to $\ker(d\lambda|_{Z_1})$ is said to be \textbf{characteristic}.
\end{definition}
In light of Proposition \ref{prop:Hsu}, we are only interested in characteristic vectors based away from the zero section $M \subset Z_1$.

Fix a ball $U \subset M$, which we may assume is an Ehresmann chart for $\xi$. Each of the bundles $Z_n$ is trivial over $U$. As such, we may assume that we have a coframe $\{\alpha_i\}_{i=1,\cdots,i_1}$ of $Z_1$ such that the $\{\alpha_i\}_{i=1,\cdots,i_n}$ span $Z_n$. Such a coframe provides coordinate functions $(a_i)_{i=1,\cdots,i_1}$ in the fibres of $Z_1$. In these coordinates we have:
\begin{align} \label{eq:Liouville}
\lambda|_{Z_1}  \quad=\quad& \sum_{i=1}^{i_1} a_i \alpha_i \notag\\
d\lambda|_{Z_1} \quad=\quad& \sum_{i=1}^{i_1} da_i \wedge\alpha_i + a_id\alpha_i,
\end{align}
where we are abusing notation and still denoting by $\alpha_i$ its pullback to $Z_1$.

\begin{lemma}\label{lem:lifting}
Let $\pi: Z_1 \to M$ denote the natural projection. Fix a characteristic vector $v \in \ker(d\lambda|_{Z_1})$ based at $(q,\alpha)$. Then:
\begin{itemize}
\item $d\pi(v)$ is tangent to $\xi_q$,
\item $v$ is the unique vector contained in $\ker(d\lambda|_{Z_1})$, based at $(q,\alpha)$, and projecting to $d\pi(v)$.
\end{itemize}
\end{lemma}
\begin{proof}
Working in the local trivialisation with fibre coordinates $(a_i)_{i=1,\cdots,i_1}$ as above, we write $v = v_v + v_h$, where $v_v$ is vertical (i.e. tangent to the fiber) and $v_h$ is horizontal. Evaluating this expression with $d\lambda|_{Z_1}$ we obtain the following $1$-form in $Z_1$:
\begin{equation} \label{eq:lifting}
0 = \iota_v d\lambda|_{Z_1} = \sum_i -\alpha_i(v_h)da_i + da_i(v_v)\alpha_i  + a_i \iota_{v_h}d\alpha_i.
\end{equation}
The first term consists of vertical $1$-forms (i.e. vanishing on horizontal vectors), whereas the other two are horizontal. If the first term is to vanish, all the coefficients $\alpha_i(v_h) = \alpha_i(d\pi(v))$ must be zero, proving the first claim.

We now inspect the other two terms. It must hold that:
\[ \sum_i da_i(v_v)\alpha_i  + a_i \iota_{d\pi(v)}d\alpha_i = 0. \]
We first note that $\sum_i da_i(v_v)\alpha_i$ is a linear combination of the coframing $(\alpha_i)$ of $Z_1$, with coefficients $da_i(v_v)$. As such, $\beta := \sum_i a_i \iota_{d\pi(v)}d\alpha_i$ must also be a $1$-form in $Z_1$. It follows that the coefficients $da_i(v_v)$ are uniquely determined by $d\pi(v)$, because they must be the coordinates of $-\beta$ in terms of the coframing.
\end{proof}

We think of this statement as follows: $\ker(d\lambda|_{Z_1})$ provides at each point $(q,\alpha) \in Z_1$ a vector subspace of $T_{(q,\alpha)}Z_1$. These subspaces have different ranks and, as such, they may be regarded as a distribution with singularities (do note that this is different from a differential system, which is the dual notion). The Lemma tells us that $\ker_{(q,\alpha)}(d\lambda|_{Z_1})$ is a \emph{lift} of a vector subspace of $\xi_q$ and, as such, it resembles an Ehresmann connection which is defined only for some horizontal directions:
\begin{corollary} \label{cor:characterisationZ}
The characteristic directions $\ker_{(q,\alpha)}(d\lambda|_{Z_1})$ are a lift of $\ker_q(d\alpha|_\xi)$. In particular:
\begin{itemize}
\item $d\lambda|_{Z_1}$ at $(q,\alpha)$ and $d\alpha|_\xi$ at $q$ have the same corank.
\item $\ker_{(q,\alpha)}(d\lambda|_{Z_1})$ lifts $\xi_q$ completely if and only if $\alpha \in Z_2$.
\item There is a bound $\ker(d\lambda|_{Z_1}) \leq \rank(\xi)$. Equality holds along $Z_2$.
\end{itemize}
\end{corollary}
\begin{proof}
All the statements follow by putting Lemmas \ref{lem:characterisationZ} and \ref{lem:lifting} together.
\end{proof}

\subsection{The key remark}

We are ultimately interested in characteristic curves of $Z_1$. As advanced in Subsection \ref{ssec:overview}, we will stratify $Z_1$ using $d\lambda|_{Z_1}$, study the curves tangent to each stratum, and then take closures to account for those curves that move across strata.

Lemma \ref{lem:characterisationZ} and Corollary \ref{cor:characterisationZ} state that, along $Z_2$, every vector in $\xi$ can be lifted to $\ker(d\lambda|_{Z_1})$. This may seem problematic, since we want to prove that characteristic/singular jets form a much smaller set than horizontal ones.

The key Linear Algebra observation is: Even though every vector in $\xi$ can be lifted to $\ker(d\lambda|_{Z_1})$ along $Z_2$, the resulting lift is not necessarily tangent to $Z_2$; this is only the case if the vector is also characteristic for $Z_2$. The same holds for higher annihilators in the flag:
\begin{proposition} \label{prop:liftingHigherZ}
The following identity holds for characteristic vectors:
\[ \ker(d\lambda|_{Z_1}) \cap TZ_n = \ker(d\lambda|_{Z_n}) \cap d\pi^{-1}(\xi). \]
\end{proposition}
\begin{proof}
First we address the inclusion $\subset$. According to Lemma \ref{lem:lifting}, $\ker(d\lambda|_{Z_1})$ is a lift of (part of) $\xi$, showing that it is contained in $d\pi^{-1}(\xi)$. The inclusion $\ker(d\lambda|_{Z_1}) \cap TZ_n \subset \ker(d\lambda|_{Z_n})$ is automatic because $d\lambda|_{Z_n}$ is the restriction of $d\lambda|_{Z_1}$ to $Z_n \subset Z_1$.

For the opposite inclusion $\supset$ we use Equation \ref{eq:lifting}. Recall that the coframing $(\alpha_i)_{i=1,\cdots,i_1}$ is adapted to the dual flag, so $(\alpha_i)_{i=1,\cdots,i_n}$ is a coframing of $Z_n$. Suppose $v = v_h + v_v \in \ker(d\lambda|_{Z_n}) \cap d\pi^{-1}(\xi)$. Then:
\begin{equation*}
\iota_v d\lambda|_{Z_1} = \sum_{i=1}^{i_1} -\alpha_i(v_h)da_i + [da_i(v_v)\alpha_i  + a_i \iota_{v_h}d\alpha_i]
\end{equation*}
splits again as a vertical and a horizontal term. The former vanishes due to the inclusion $v \in d\pi^{-1}(\xi)$. We must show that the other one does as well.

The horizontal term reads:
\begin{equation} \label{eq:Zn}
\sum_{i=1}^{i_1} da_i(v_v)\alpha_i  + a_i \iota_{v_h}d\alpha_i = \sum_{i=1}^{i_n} da_i(v_v)\alpha_i  + a_i \iota_{v_h}d\alpha_i,
\end{equation}
where we have used the inclusion $v \in \ker(d\lambda|_{Z_n}) \subset TZ_n$ to deduce that $a_i, da_i(v_v) \neq 0$ only if $i=1,\cdots,i_n$. As such, the horizontal term is exactly the same as the horizontal term for a characteristic vector of $d\lambda|_{Z_n}$, and it therefore vanishes by assumption. 
\end{proof}
Our proof of Theorem \ref{thm:main2} can then be understood as being inductive (even though we will not quite phrase it as such): We assume that we have addressed the characteristic curves contained in $Z_n$ (using as inductive hypothesis that $\xi_n$ is bracket-generating of one step less) so that we can focus on those passing through $Z_{n-1} \setminus Z_n$.

\section{Recollections of analytic geometry} \label{sec:analyticGeometry}

The subspaces of jets that we consider in this paper will be controlled thanks to the analyticity of $\xi$. We recall now some of the definitions that play a role in these constructions. Their main purpose is justifying that these spaces of $r$-jets are suitably stratified by submanifolds and can be endowed with a proper notion of codimension.

\subsection{Semianalytic sets}

We now recall results from \cite{bierstone} mostly, but we refer the reader to \cite{coste,gabrielov,goresky,hardt1,hardt2,whitney} as well. Following \cite[Definition 2.1]{bierstone} we define:
\begin{definition}
Let $W$ be a real analytic manifold and $U \subset W$ an open. We write $\OO(U)$ for the ring of real analytic functions on $U$ and $S(\OO(U))$ for the family of subsets of $U$ generated by
\[ \{ x \ | \ f(x) > 0,\, f \in \OO(U)\} \] 
and closed under finite intersection, finite union, and complement.

A subset $X$ of $W$ is \textbf{semianalytic} if each $x \in W$ has a neighborhood $U$ such that $X \cap U \in S(\OO(U))$. We will say that $X$ is \emph{smooth} if it is a smooth submanifold (without boundary) of $W$.
\end{definition}

We then define:
\begin{definition} \label{def:stratification0}
Let $W$ be a real analytic manifold. A \textbf{stratification} $\SS$ of $W$ is a partition into smooth semianalytic sets, called \textbf{strata}, such that:
\begin{itemize}
\item every point has a neighborhood intersecting only finitely many strata,
\item the \textbf{frontier} $\overline{S} \setminus S$ of a stratum $S \in \SS$ is a union of other strata.
\end{itemize}
\end{definition}
We remark that the inclusion of a stratum $S'$ in the frontier of another stratum $S$ defines a partial order $S' < S$. To construct our stratifications we will need to invoke the following result \cite[Corollary 2.8]{bierstone}:
\begin{proposition} \label{prop:closureFrontier}
The closure and the frontier of a semianalytic set are semianalytic.
\end{proposition}

For the purpose of stratifying our spaces of jets into smooth pieces we will need \cite[Proposition 2.10 and Corollary 2.11]{bierstone} as well:
\begin{proposition} \label{prop:stratifySmooth}
Let $W$ be a real analytic manifold. Any locally finite family of semianalytic subsets $(A_j)_j$ of $W$ can be refined to a stratification of $W$ (i.e. each $A_j$ will be a union of strata).
\end{proposition}
This Proposition tells us that we can speak of the dimension of any semianalytic subset $A \subset W$. Indeed, this is the maximal dimension among the smooth strata in which it can be decomposed.

\subsection{Subanalytic sets} \label{ssec:subanalytic}

Our arguments in Sections \ref{sec:stratification} and \ref{sec:localIntegrability} require us to work with semianalytic sets in $J^r(Z_1)$, which we then project down to $J^r(M)$. However, unlike semialgebraic sets, semianalytic sets are not closed under projection. This leads us to the definition of subanalytic set \cite[Definition 3.1]{bierstone}:
\begin{definition}  \label{def:subanalytic}
Let $W$ be a real analytic manifold. A subset $X$ of $W$ is \textbf{subanalytic} if every $p \in W$ admits a neighborhood $U$ such that $X \cap U$ is the projection of a relatively compact semianalytic set.

That is: There exists a real analytic manifold $V$ and a relatively compact semianalytic subset $A \subset W \times V$ such that $X \cap U = \pi(A)$, where $\pi: W \times V \rightarrow W$ is the projection.
\end{definition}

These spaces have a notion of dimension, which is the maximal dimension computed at any of its smooth points. This follows from the fibre-cutting Lemma \cite[Lemma 3.6]{bierstone}:
\begin{proposition} \label{prop:subanalyticStrat1}
Let $W$, $X$, $U$, $V$, $A$, $\pi$ be as in the previous definition. Then, there exists a finite collection $(B_i)_i$ of smooth semianalytic sets in $W \times V$, contained in $A$, such that $\pi|_{B_i}$ is an immersion and $X \cap U = \cup_i \pi(B_i)$.
\end{proposition}
That is: A subanalytic subset $X$ is the union of a finite collection of immersed submanifolds whose dimensions are bounded above by the dimensions of the semianalytic sets $A$ that locally project down to $X$. This collection is not quite a stratification of $X$, but this dimension control is sufficient for us.

\subsection{Some technical Lemmas} \label{ssec:technicalLemmas}

We now explain how one may partition a manifold into semianalytic subsets that are nicely adapted to singular distributions and 2-forms. This will play a key role in the next Section. These statements boil down to elementary Linear Algebra, and probably follow from some general stratification result involving analytic sections of jet spaces (in the spirit of Thom-Boardman), but we do not know of an appropriate reference.

The first ingredient, from which the subsequent claims follow, reads:
\begin{lemma} \label{lem:stratifyRankGeneral}
Let $E \to W$ be a real analytic vector bundle with connected base. Let $\{e_i\}$ be a finite collection of sections of $E$. Then, $W$ admits a partition into the semianalytic subsets
\[ A_j := \{ q \in W \,\mid\, \rank(\langle e_1,e_2,\cdots \rangle) = j \}. \]
If $W$ is connected, there is a (largest) $j_0$ such that $A_{j_0}$ is dense in $W$.
\end{lemma}
\begin{proof}
We work locally, allowing us to assume that $E$ is trivial. We can then build a matrix $B$ whose columns are the local coefficients of the vectors $\{e_i\}$. The entries of $B$ are analytic functions, and therefore so are its minors. The set $A_j$ is given by the vanishing of all $(j+1)$-minors of $B$ and the non-vanishing of a $j$-minor. It is thus semianalytic.

We now justify the second claim: If $j_0$ is the largest $j$ with $A_j$ non-empty, that means that all $(j_0+1)$-minors vanish identically in $W$ but some $j_0$-minor does not. The zero set of all $j_0$-minors is $A_{j-1}$; it must have positive codimension in $W$ because, due to connectedness, a non-zero analytic function cannot vanish over an open of $W$.
\end{proof}

\begin{corollary} \label{cor:stratifyRankPffaff}
Let $W$ be a real analytic manifold endowed with a finite collection of analytic $1$-forms $\{\alpha_i\}$. Then, any smooth semianalytic subset $V \subset W$ admits a partition into the semianalytic subsets
\[ A_j := \{ q \in V \,\mid\, \rank(T_qV \cap \ker(\alpha_1,\cdots)) = j \}. \]
If $V$ is connected, there is a (smallest) $j_0$ such that $A_{j_0}$ is dense in $V$.
\end{corollary}
Before we start, let us remark that the key point of the proof (and the proofs of the claims that follow) is to write the conditions we are interested in (which are defined only over $V$) as conditions in terms of (locally defined) analytic functions in $W$. This is what provides semianalyticity.
\begin{proof}
We work locally. Being a smooth semianalytic submanifold, $V$ is locally described by a finite collection of equations $\{f_k \,\Box\, 0\}_{k \in K}$, where $\Box$ denotes $=$, $<$, or $>$. Let $K_0 \subset K$ be the subcollection indexing those functions $f_k$ that vanish on $V$. 

We apply Lemma \ref{lem:stratifyRankGeneral} to yield a partition $\SS$ of $W$ into subsets according to the rank of $\langle df_k,\alpha_i\rangle_{k \in K_0,i}$. Along $V$, this is exactly the corank of $TV \cap \ker(\{\alpha_i\})$ in $TW$. We can then intersect the strata of $\SS$ with $V$ to yield the claimed partition of $V$ into semianalytic sets.
\end{proof}

We may prove the analogous result for $2$-forms:
\begin{corollary}\label{cor:stratifyRank2Form}
Let $W$ be a real analytic manifold endowed with an analytic $2$-form $\omega$. Then, any smooth semianalytic subset $V \subset W$ admits a partition into the semianalytic subsets
\[ A_j := \{ q \in V \,\mid\, \rank(\omega|_{T_qV}) = j \}. \]

If $V$ is connected, there is a (smallest) $j_0$ such that $A_{j_0}$ is dense in $V$.
\end{corollary}
\begin{proof}
Again, we describe $V$ locally using expressions of the form $\{f_k \,\Box\, 0\}_{k\in K}$. We assume that the subcollection $K_0$ (corresponding to $\Box$ being an equality) is minimal in the sense that the differentials $\{df_k\}_{k\in K_0}$ are linearly independent. We also fix a local analytic framing $(w_i)_i$ of $TW$.

The key objects to study are the analytic $(k_0+2)$--form 
\[ \Omega := df_1 \wedge \cdots \wedge df_{k_0} \wedge \omega, \]
and the associated ($k_0+1$)--forms $\beta_i := \iota_{w_i} \Omega$. The point is that $\Omega$ is a (local) object in $W$ which, along $V$, represents the restriction $\omega|_{TV}$. Thus, $\rank(\omega|_{TV})$ will be encoded in the rank of $B := \langle\beta_i\rangle_i$, as we now explain.

First: Note that, along $V$, wedging with $df_1 \wedge \cdots \wedge df_{k_0}$ kills any form restricting to zero on $TV$. As such, $\Omega$ only depends on $\omega|_{TV}$. In particular, $\Omega = 0$ and $\rank(B) = 0$ if and only if $\omega|_{TV} = 0$. We henceforth assume otherwise.

Second: $\rank(B)$ and $\rank(\omega|_{TV})$ are pointwise quantities, so we can compute them at each point $q \in V$. Fix a splitting $T_qV \oplus E_q = T_qW$; we say the vectors tangent to $T_qV$ are horizontal and the ones tangent to $E_q$ are vertical. Then we can write:
\[ \Omega_q = d_qf_1 \wedge \cdots \wedge d_qf_{k_0} \wedge \widetilde\omega_q \]
with $\widetilde\omega_q$ annihilating $E_q$. Instead of using the framing $(w_i)_i$, we compute $B$ using a basis of $T_qW$ adapted to this splitting. Namely, we pick horizontal $(v_j \in T_qV)_j$ and vertical $(e_k \in E_q)_k$ bases; we require $e_k$ to be dual to $d_qf_k$.

It follows that, at $q$, $B$ is spanned by:
\begin{itemize}
\item The forms 
\[ \iota_{e_l}\Omega_q = d_qf_1 \wedge \cdots \wedge \widehat{d_qf_l} \wedge \cdots \wedge d_qf_{k_0} \wedge \widetilde\omega_q.  \]
These span a $k_0$-dimensional subspace. This follows from the fact that $\widetilde\omega_q$ is horizontal and the $d_qf_k$ are linearly independent.
\item Together with the forms
\[ \iota_{v_j}\Omega_q = d_qf_1 \wedge \cdots \wedge d_qf_{k_0} \wedge (\iota_{v_j} \widetilde\omega_q). \]
These form a subspace of dimension $\rank(\widetilde\omega|_{T_qV}) = \rank(\omega|_{T_qV})$.
\end{itemize}
As such, $\rank(B) = \rank(\omega|_{TV})+k_0$ (unless $\Omega$ is zero). The result then follows applying Lemma \ref{lem:stratifyRankGeneral} to the collection $\{\beta_i\}$ and intersecting the resulting strata with $V$.
\end{proof}

We may also prove a relative version of the previous Corollary, in which the two entries of the 2-form are tangent to different subsets:
\begin{corollary} \label{cor:stratifyRank2FormB}
Let $W$ be a real analytic manifold endowed with an analytic $2$-form $\omega$. Let $V \subset W$ be a smooth semianalytic subset. Then, any smooth semianalytic subset $Y \subset W$ contained in $V$ admits a partition into semianalytic subsets:
\[ A_j := \{ q \in Y \,\mid\, \rank(\ker(\omega|_{TV}) \cap T_qY) = j \}. \]

If $V$ is connected, there is a $j_0$ such that $A_{j_0}$ is dense in $Y$.
\end{corollary}
\begin{proof}
Let $\{f_k \,\Box\, 0\}_{k\in K}$ be local functions defining $V$ and let $\{g_l \,\Box\, 0\}_{l\in L}$ be an additional family of functions such that both together define $Y$. We assume that the subcollections $\{f_k\}_{k\in K_0}$ and $\{f_k,g_l\}_{k\in K_0, l \in L_0}$ are minimal.

From the $(k_0+2)$-form $\Omega := df_1 \wedge \cdots \wedge df_{k_0} \wedge \omega$ we define the ($k_0+l_0+1$)--forms 
\[ \beta_i := dg_1 \wedge \cdots \wedge dg_{l_0} \wedge \iota_{w_i} \Omega, \]
where $(w_i)_i$ is an analytic framing of $TW$. These forms represent pairing a vector in $TV$ with a vector in $TY$ using $\omega$. We reason as in the previous Corollary but we omit the discussion involving a linear splitting.

The $B := \langle \beta_i \rangle_i$ is zero dimensional if and only if $\ker(\omega|_{TV}) \supset TY$. Otherwise, $B$ is spanned by:
\begin{itemize}
\item The $k_0$-dimensional subspace spanned by the forms 
\[ dg_1 \wedge \cdots \wedge dg_{l_0} \wedge df_1 \wedge \cdots\wedge \widehat{df_l} \wedge \cdots \wedge df_{k_0} \wedge (\omega|_{TV}). \]
\item The subspace spanned by the forms
\[ dg_1 \wedge \cdots \wedge df_{k_0} \wedge (\iota_v \omega) \qquad\textrm{with } v \in TV, \]
which correspond to the $1$-forms $(\iota_v \omega)|_{TY}$.
\end{itemize}
It follows that $\rank(B)-k_0$ is the corank of $\ker(\omega|_{TV}) \cap TY$ in $TY$. Applying Lemma \ref{lem:stratifyRankGeneral} to $B$ and intersecting with $V$ allows us to conclude.
\end{proof}

We will also need a combination of Corollaries \ref{cor:stratifyRankPffaff} and \ref{cor:stratifyRank2FormB} which, in fact, subsumes them:
\begin{corollary} \label{cor:stratifyRankFinal}
Let $W$ be a real analytic manifold, $\omega$ an analytic $2$-form, and $\{\alpha_i\}$ a finite collection of analytic $1$-forms. Let $V \subset W$ be a smooth semianalytic subset. Then, any smooth semianalytic subset $Y \subset W$ contained in $V$ admits a partition into semianalytic subsets:
\[ A_j := \{ q \in Y \,\mid\, \rank(\ker(\omega|_{TV}) \cap T_qY \cap \ker(\alpha_1,\cdots)) = j \}. \]

If $V$ is connected, there is a $j_0$ such that $A_{j_0}$ is dense in $Y$.
\end{corollary}
\begin{proof}
We fix families $\{f_k \,\Box\, 0\}_{k\in K}$ and $\{g_l \,\Box\, 0\}_{l\in L}$ as in the previous Corollary.

The added difficulty now is that the collection $\{\alpha_i\}$ is not necessarily linearly independent (by itself, or together with the $\{df_k\}_{k \in K_0}$ or $\{dg_l\}_{l \in L_0}$). As such we first invoke Lemma \ref{lem:stratifyRankGeneral} to partition $V$ depending on the rank of $\langle \alpha_i, df_k, dg_l \rangle_{i,k \in K_0, l \in L_0}$; we write $B_r$ for the locus with rank $r$.

Once we have done that, we proceed similarly to Corollary \ref{cor:stratifyRank2FormB}, but we have to do it simultaneously in different degrees. Let us elaborate: We fix a local framing $\{w_a\}$ of $TW$ and we look at the forms:
\[ \beta_{a,I} := \alpha_{i_1} \wedge \cdots \wedge \alpha_{i_{|I|}} \wedge dg_1 \wedge \cdots \wedge \iota_{w_a}(df_1 \wedge \cdots \wedge \omega). \]
Where $I = (a_{i_1} < \cdots < a_{i_{|I|}})$ ranges over all subsets of the index set of $\{\alpha_i\}$. We observe that, along $B_r$, $\rank(\langle \beta_{a,I} \rangle_{|I| = r, a})$ is precisely the corank of $\ker(\omega|_{TV}) \cap TY \cap \ker(\alpha_1,\cdots)$ within $TY$.

Thus, the claim follows by partitioning $B_r$ in terms of the rank of $\langle \beta_{a,I} \rangle_{|I| = r, a}$ and taking suitable unions of the resulting sets.
\end{proof}

Lastly, we can address the dual case:
\begin{corollary}\label{cor:stratifyRankVector}
Let $W$ be a real analytic manifold endowed with a finite collection of analytic vector fields $\{w_i\}$. Then, any smooth semianalytic subset $V \subset W$ admits a partition into the semianalytic subsets
\[ A_j := \{ q \in V \,\mid\, \rank(T_qV \cap \langle w_i\rangle_i) = j \}. \]

If $V$ is connected, there is a (smallest) $j_0$ such that $A_{j_0}$ is dense in $V$.
\end{corollary}
\begin{proof}
We can evaluate the $(w_i)$ in the differentials $\{df_k\}_{k\in K_0}$ that annihilate $TV$. In doing so we may form a matrix with entries $\{df_k(w_i)\}_{i,k}$. Its rank at each point of $V$ is the dimension of $\langle w_i\rangle_i/TV$. Applying Lemma \ref{lem:stratifyRankGeneral} we partition $V$ in terms of this rank. We may then refine this partition by looking at the rank of $\langle w_i \rangle_i$ instead. The difference between these two numbers is precisely the rank of $TV \cap \langle w_i \rangle_i$. We deduce that the set $A_j$ is a union of semianalytic sets (and thus semianalytic).
\end{proof}

\section{Stratifying the annihilator} \label{sec:stratification}

We work with $(M,\xi)$ analytic and bracket-generating. In this Section we describe a stratification of $Z_1 = \Ann(\xi)$ that is nicely adapted to $d\lambda$ (equivalently, adapted to the rank of the curvature(s) associated to $\xi$). Our discussion could easily be adapted to the non-bracket-generating case, but this is unnecessary for our purposes.

First we work under regularity assumptions (Subsection \ref{ssec:stratificationR}), and then we adapt our approach to the general case (Subsection \ref{ssec:stratificationNR}). In Subsection \ref{ssec:dimensionCounting}, the bracket-generating assumption will provide us with certain key dimension bounds.

\subsection{The regular case} \label{ssec:stratificationR}

We write $n_0$ for the index in which the Lie flag stabilises. Then the dual flag reads:
\[ Z_1 \supset Z_2 \supset Z_3 \supset ... \supset Z_{n_0}. \]
Note that $Z_{n_0}$ is the zero section $M$ precisely because $\xi$ is bracket-generating. For notational ease we set $Z_{n_0+1} = \emptyset$.

We are interested in stratifications satisfying the following properties:
\begin{definition} \label{def:stratification}
Let $(M,\xi)$ be analytic, bracket-generating, and regular. A stratification $\SS$ of $Z_1 = \Ann(\xi)$ is said to be \textbf{Ehresmann-Liouville} if:
\begin{itemize}
\item For every $S \in \SS$, there is $n$ such that $S \subset Z_n \setminus Z_{n+1}$.
\item For every $S \in \SS$ and $n$ as above, the intersection 
\[ \Xi_S := TS \cap \ker(d\lambda|_{Z_n}) \cap d\pi^{-1}(\xi) \]
is a distribution of constant rank.
\item The strata $S \in \SS$, as well as the corresponding $\Xi_S$, are invariant under the scaling $\eta$-action.
\end{itemize}
\end{definition}
The reasoning behind the definition of $\Xi_S$ is Proposition \ref{prop:liftingHigherZ}: It states that, along $Z_n$, it is enough to study those curves that project to $\xi$ and are characteristic for $Z_n$.

Separately: From the last item we deduce that the restriction $\SS \setminus \{M\}$ of an Ehresmann-Liouville stratification $\SS$ to $Z_1 \setminus M$ is equivalent to a stratification $\P\SS$ of $\P(Z_1)$ (and, in practice, it is often more convenient to think of the latter). Recall that, following Proposition \ref{prop:Hsu}, we are only interested in what happens away from the zero section.

\begin{proposition}\label{prop:stratification}
Let $(M,\xi)$ be real analytic, bracket-generating, and regular. Then, its annihilator $Z_1$ admits an Ehresmann-Liouville stratification $\SS$. 
\end{proposition}
\begin{proof}[Proof of Proposition \ref{prop:stratification}]
We restrict $d\lambda$ to $Z_n$. We will work simultaneously on all pairs $(Z_n,d\lambda|_{Z_n})$, producing a stratification of $Z_1$. We will use the results from Subsection \ref{ssec:technicalLemmas} repeatedly.

We work locally in $M$. As such, we may choose an analytic coframing $A_n$ of $\xi_n$. Then, over $Z_n$, we consider the (singular) distribution
\[ \Xi_{Z_n} := \ker(d\lambda|_{Z_n}) \cap d\pi^{-1}(\xi). \]
It is defined as the common kernel in $TZ_n$ of the 2-form $d\lambda$ and the coframing $A_n$ (pulled back to $Z_n$). An application of Corollary \ref{cor:stratifyRankFinal} then yields a partition 
\[ \SS_1^n := \left(S_{n,j}^{(1)}\right)_{j \in J_n^{(1)}} \]
of $Z_n$ into semianalytic subsets such that $\Xi_{Z_n}$ has constant rank $j$ over $S_{n,j}^{(1)}$. The $S_{n,j}^{(1)}$ are not necessarily smooth.

Note that the piece $S_{n,\rank(\xi)}^{(1)}$ is precisely $Z_{n+1}$, which we have partitioned according to the rank of $\Xi_{Z_{n+1}}$ instead. Additionally, according to Equation \ref{eq:Liouville}, the distribution $\Xi_M$ associated to the zero section $Z_{n_0} = M$ is $\xi$. We will henceforth not pay attention to $Z_{n_0}$.

We define $\SS_1$ to be the partition of $Z_1$ which in each $Z_n$ is given by $\SS_1^n$. Proposition \ref{prop:stratifySmooth} refines $\SS_1$ to a stratification $\SS_2$ by smooth semianalytic subsets. Because $\SS_1$ was a refinement of the dual flag, so is $\SS_2$; we denote by $\SS_2^n$ the induced stratification of $Z_n$. We may iterate this process: We apply Corollary \ref{cor:stratifyRankFinal} to each triple $(S \in \SS_{2l}^n,d\lambda|_{Z_n},A_n)$. This yields a partition of $S$ into semianalytic subsets according to the rank of $\Xi_S := TS \cap \Xi_{Z_n}$. Since we do it for all $S \in \SS_{2l}$, we obtain a locally finite partition $\SS_{2l+1}$ of $Z_1$. We refine it again to a stratification $\SS_{2l+2}$ using Proposition \ref{prop:stratifySmooth}.

We claim that this process terminates in finitely many steps (say, $l_0$), yielding a stratification $\SS := \SS_{2l_0}$ whose strata
\[ S \in \SS^n := \SS_{2l_0}^n = \left(S_{n,j}^{(2l_0)}\right)_{j \in J_n^{(2l_0)}} \]
have constant rank distributions $\Xi_S = TS \cap \Xi_{Z_n}$. Indeed: any infinite sequence of semianalytic subsets must necessarily stabilise in dimension, and thus also in rank of the distribution.

Lastly, we claim that $\SS$ is $\eta$-invariant. This follows from the fact that $d\lambda$ is homogeneous, so its kernel is dilation invariant. In particular, we could have stratified $\P(Z_1)$ instead, following the same steps, and then take preimages.
\end{proof}

\begin{remark}
In the smooth setting, we cannot possibly expect the rank of $d\lambda|_{Z_n}$ to cut out smooth submanifolds and, indeed, one can construct examples where Cantor sets arise as the decomposing sets $S$. However, as we pointed out in Remark \ref{rem:genericCase}, one expects Thom-Boardman transversality to produce a stratification of $Z_1$ if $\xi$ is smooth and generic.
\end{remark}

\subsection{The non-regular case} \label{ssec:stratificationNR}

We will now adapt the previous definition and proof to the non-regular case. The main idea is to first stratify the base manifold $M$ to yield a dense piece in which regularity holds, and then stratify $Z_1$ over that piece as before. Over the other pieces we will not have to be particularly careful: the fact that they are lower dimensional is enough for our dimension counting arguments in Subsection \ref{ssec:dimensionCounting}.

\begin{definition} \label{def:stratificationBaseNonRegular}
Let $(M,\xi)$ be analytic, bracket-generating, possibly not regular. A pair of stratifications $\SS_M$ (of $M$) and $\SS$ (of $Z_1 = \Ann(\xi)$) is \textbf{Ehresmann-Liouville} if:
\begin{itemize}
\item For each $S \in \SS_M$, the pointwise rank of $TS \cap \Gamma(\xi)^{(n)}$ is constant and it thus arises from a distribution $\xi_{S,n} \subset TS$.
\item $\SS$ is subordinated to the preimage in $Z_1$ of $\SS_M$.
\item When restricted to an open strata of $\SS_M$, the stratification $\SS$ is Ehresmann-Liouville. 
\end{itemize}
\end{definition}
We write $\SS^o,\SS^c_M \subset \SS_M$ for the open/non-open strata. Similarly, we write $\SS^o,\SS^c \subset \SS$ for those strata lying over $\SS^o$ and $\SS^c_M$.

\begin{proposition} \label{prop:stratificationNR}
Let $(M,\xi)$ be analytic and bracket-generating. Then $\Ann(\xi)$ admits an Ehresmann-Liouville pair $(\SS_M,\SS)$.
\end{proposition}
\begin{proof}
We work locally in $M$, allowing us to trivialise $\xi$. Since $\xi$ admits a framing $A_1$, each $\Gamma(\xi)^{(n)}$ is finitely generated by brackets with entries in said framing; we denote this generating set by $A_n$. We may then apply Corollary \ref{cor:stratifyRankVector} to each $A_n$ to partition $M$. We then intersect all these partitions and apply Proposition \ref{prop:stratifySmooth} to yield a stratification by smooth semianalytic sets. We iterate this process: we first apply Corollary \ref{cor:stratifyRankVector} to $A_n$ and each stratum of the previous stratification and then Proposition \ref{prop:stratifySmooth}. After finitely many steps the process terminates, yielding $\SS_M$.

Let $S \in \SS_M^o$ be an open stratum (there is one in each connected component due to analyticity). To construct $\SS$, we follow the proof of Proposition \ref{prop:stratification} verbatim with $Z_1|_S$ instead of $Z_1$. The key points are:
\begin{itemize}
\item $Z_1|_S$ is a semianalytic set of $Z_1$ (indeed, it is cut out by pulling back the equations of $S$). As such, we may apply Corollary \ref{cor:stratifyRankFinal} to $d\lambda$, the pullback of a coframing of $\xi$, and $Z_1|_S$. This partitions $Z_1|_S$ into sets that are semianalytic globally in $Z_1$.
\item The set where $d\lambda|_{Z_1}$ has minimal rank is $Z_2|_S$. Being the locus of minimal rank, it is semianalytic and, by our assumptions on $S$, it is a vector bundle of constant rank over $S$. We can apply Corollary \ref{cor:stratifyRankFinal} to stratify it as in the regular case. Iterating this procedure defines the higher annihilators $Z_n|_S$ and their partitions.
\item Once we have partitioned $Z_1|_S$ using this inductive procedure on $n$, we apply Proposition \ref{prop:stratifySmooth} to the whole of $Z_1$. When we do so, we also take into account the partition of $Z_1$ induced by $\SS_M$ by taking preimages. This refines all these partitions to a global stratification.
\end{itemize}
As in the regular case, we repeat this scheme until no further refinements are needed because the distributions $\Xi_S$ have constant rank. The claimed properties hold by construction.
\end{proof}

\subsection{Dimension counting} \label{ssec:dimensionCounting}

Suppose $\xi$ is regular. Fix a stratum $S \in \SS$. According to Lemma \ref{lem:lifting}, $\ker(d\lambda|_{Z_n})$ is a (partial) lift of $\xi_n$, so
\[ \rank(\ker(d\lambda|_{Z_n})) \leq \rank(\xi_n). \]
It follows that $\ker(d\lambda|_{Z_n}) \cap d\pi^{-1}(\xi)$, which contains $\Xi_S$, is a partial lift of $\xi$, so the analogous bound for their ranks holds.

The following more refined bound, which we will exploit, depends on the bracket-generating condition:
\begin{lemma} \label{lem:dimensionCounting}
Let $(M,\xi)$ be regular, analytic, and bracket-generating. Let $\SS$ be an Ehresmann-Liouville stratification. Then 
\[ \rank(\Xi_S) < \rank(\xi) \]
for every stratum $S \in \SS$ not contained in the zero section.
\end{lemma}
\begin{proof}
Let $\alpha \in S \subset Z_n \setminus Z_{n+1}$ based at $q \in M$. According to Lemma \ref{lem:characterisationZ}, this means that $d_q\alpha|_{\xi_n}$ (as a form in the base) is not zero (here we have chosen some extension of $\alpha$ to a local section of $Z_n$, but the concrete extension is not important). Due to the Jacobi identity, this implies that $\xi$ is not completely contained in $\ker(d\alpha|_{\xi_n})$. It follows from Corollary \ref{cor:characterisationZ} that $\ker_\alpha(d\lambda|_{Z_n}) \cap (d_\alpha\pi)^{-1}(\xi_q)$ is a lift with strictly less rank than $\xi_q$. The same follows then for $(\Xi_S)_\alpha$.

Now we use the bracket-generating condition: Our reasoning applies in each $Z_n \setminus Z_{n+1}$ with $n < n_0$. If $\xi$ is bracket-generating, every $\alpha \neq 0$ is contained in such a $Z_n$.
\end{proof}

We obtain similar dimension bounds in the non-regular case:
\begin{lemma} \label{lem:dimensionCountingNR}
Let $(M,\xi)$ be analytic and bracket-generating, with Ehresmann-Liouville pair $(\SS_M,\SS)$. Then:
\begin{itemize}
\item $\rank(\xi \cap TS) < \rank(\xi)$ for every stratum $S \in \SS_M^c$ other than the open ones.
\item $\rank(\Xi_S) < \rank(\xi)$ for every stratum $S \in \SS^o$ not contained in the zero section.
\end{itemize}
\end{lemma}
\begin{proof}
For the first claim we just need to observe that $S$ has positive codimension in $M$. By construction, $\xi \cap TS$ is a smooth distribution in $S$. The bracket-generating condition then implies that it must be strictly smaller than $\xi$.

The second claim follows directly by applying Lemma \ref{lem:dimensionCounting} to the open strata.
\end{proof}

\section{Local integrability}\label{sec:localIntegrability}

We are now ready to analyse the codimension of the locus of inadmissible jets within the space of all horizontal jets. This will allow us to prove Theorem \ref{thm:main2}. We rely on the results from the two previous Sections.

We work with $(M,\xi)$ analytic and bracket-generating (but possibly non-regular). We write $Z_1 := \Ann(\xi)$ for its annihilator. We apply Proposition \ref{prop:stratificationNR} to obtain an Ehresmann-Liouville pair $(\SS_M,\SS)$ adapted to $\xi$.

\subsection{Jets tangent to semianalytic sets} \label{ssec:subsetJets}

First we prove an auxiliary Lemma:
\begin{lemma} \label{lem:subsetJets}
Let $S$ be a smooth semianalytic subset of an analytic manifold $W$. Then, $J^r(S)$ is a semianalytic subset of $J^r(W)$.
\end{lemma}
\begin{proof}
We work locally. Let $(f_i \,\Box\, 0)_{i \in K}$ be the defining equations of $S$, with $K_0 \subset K$ indexing the equalities.

Then, $J^r(S)$ is the subset of jets that annihilate the analytic $1$-forms $(df_i)_{i \in K_0}$ and whose basepoints lie in $S$. As such, it is the intersection of two semianalytic sets.
\end{proof}

\subsection{Characteristic jets} \label{ssec:characteristicJets}

We recall:
\begin{definition} \label{def:characteristicJets}
An $r$-jet in $J^r(Z_1)$ is \textbf{characteristic} if:
\begin{itemize}
\item $d\lambda|_{Z_1}$ vanishes to order $r$ on any of its representatives.
\item It is based in $Z_1 \setminus M$.
\end{itemize}
The space of all characteristic jets is denoted by $J^r(Z_1,d\lambda) \subset J^r(Z_1)$.
\end{definition}
We note that this does not depend on the choice of representative. The following can be proven similarly to Lemma \ref{lem:integralJets}:
\begin{lemma}
Fix $\alpha \in Z_1$. Then, $J^r_\alpha(Z_1,d\lambda)$ is an algebraic subvariety of $J^r_\alpha(Z_1)$.
\end{lemma}

We can prove a jet analogue of Hsu's theorem Proposition \ref{prop:Hsu}:
\begin{lemma}
Let $\gamma: \Op(0) \to (M,\xi)$ be a singular curve. Then $j^r\gamma \in J^r(M,\xi)$ admits a characteristic lift to $J^r(Z_1,d\lambda)$.
\end{lemma}
\begin{proof}
According to Hsu's result, we can lift $\gamma$ to a characteristic curve in $Z_1$. We then take $r$-jets at $0$ (but the reader should note that we could take jets at any other point, showing that the jet we obtain is part of a family).
\end{proof}

Recall the actions $\rho$ and $\eta$ on jets defined in Subsections \ref{sssec:rho} and \ref{sssec:eta}. Invoking that $\lambda|_{Z_1}$ is analytic:
\begin{lemma} \label{lem:characteristicJets}
$J^r(Z_1,d\lambda)$ is a $(\rho\oplus\eta)$-invariant semianalytic subvariety of $J^r(Z_1)$.
\end{lemma}
\begin{proof}
Locally in $M$, fix a framing $\{v_j\}$ of $TZ_1$. We can consider the analytic $1$-forms $\iota_{v_j}d\lambda$ and plug them into an $r$-jet of curve in $Z_1$. This yields (locally) an analytic map between jet spaces
\begin{equation}\label{eq:characteristicJets} 
J^r(Z_1) \to \left(\R^{\dim(Z_1)} \to \R \right)^{(r-1)}_0.
\end{equation}
Its zero set away from the zero section is precisely $J^r(Z_1,d\lambda)$ which is then semianalytic. Invariance is immediate from the discussion in Subsections \ref{sssec:rho} and \ref{sssec:eta}, together with the homogeneity of $d\lambda$.
\end{proof}

We can now project to the base manifold:
\begin{corollary} \label{cor:projectCharacteristicJets}
The projection map
\[ j^r\pi: J^r(Z_1) \to J^r(M) \]
is analytic. The image of $J^r(\Z_1,d\lambda)$ is subanalytic and contained in $J^r(M,\xi)$.
\end{corollary}
\begin{proof}
In local coordinates over $M$, $Z_1$ is trivialised by a coframing $\{\alpha_i\}$, yielding fibre coordinates $(a_i)$. We recall the local expression Equation \ref{eq:Liouville}:
\[ d\lambda|_{Z_1} = \sum_{i=1}^{i_1} da_i \wedge\alpha_i + a_id\alpha_i, \]
where the pullback of $\alpha_i$ to $Z_1$ is denoted in the same manner.

We may assume that the framing of $Z_1$ used in the previous Lemma is adapted to this trivialisation. That is, the framing may be decomposed into an horizontal part $(\partial_{q_j})_{j=1,\cdots,\dim(M)}$ and a vertical part $(\partial_{a_i})_{i=1,\cdots,\rank(Z_1)}$. Then, the vertical part of Equation \ref{eq:characteristicJets} reads:
\begin{align*}
J^r(Z_1) \quad\to\quad & \left(\R^{\rank(Z_1)} \to \R \right)^{(r-1)}_0 \\
\sigma  \quad\to\quad & (d\lambda(\sigma',\partial_{a_1}),\cdots).
\end{align*}
We can readily expand
\[ d\lambda(\sigma',\partial_{a_i}) = -\alpha_i(\sigma') = -\alpha_i((\pi \circ \sigma)'), \]
showing that $\pi \circ \sigma$ is tangent to each $\alpha_i$ up to order $r$  (and thus horizontal) if $\sigma$ is characteristic.

The subanalyticity claim follows from the previous Lemma \ref{lem:characteristicJets}, together with the fact that projections of (relatively compact) semianalytic sets are subanalytic (Subsection \ref{ssec:subanalytic}). However, this is not immediate in our setting, because the sets we deal with are not compact. To fix this, we work instead in the $(\rho\oplus\eta)$-quotient, which is compact (Lemma \ref{lem:equivariant}).

Indeed, $J^r(Z_1,d\lambda)/(\rho\oplus\eta)$ is semianalytic in $J^r(Z_1)/(\rho\oplus\eta)$ and closed, and thus compact. Then, its projection in $J^r(M)/\rho$ is subanalytic. Equivalently, any $\pi(J^r(Z_1,d\lambda)) \cap U$, with $U \subset J^r(M)$ compact, is the image of a compact semianalytic slice contained in $J^r(Z_1,d\lambda)$, proving the claim.
\end{proof}

Given a characteristic $r$-jet $\sigma$, it is not a priori clear that there is a characteristic curve representing it. Such a statement would hold, using the Ehresmann chart formalism of Subsection \ref{ssec:EhresmannJets}, if $d\lambda|_{Z_1}$ was of constant rank. This precisely motivates our interest in stratifying $Z_1$. We explain this next.

\subsection{Characteristic jets of tangency type} \label{ssec:tangencyJets}

Following Corollary \ref{cor:projectCharacteristicJets} we would now be inclined to study \emph{singular jets} in $J^r(M,\xi)$ as the image of $J^r(Z_1,d\lambda)$ under projection. The issue with this is that we do not actually know how to control the behaviour of \emph{all} characteristic jets. Our methods only control those that interact nicely with the Ehresmann-Liouville pair $(\SS_M,\SS)$\footnote{Due to this, we cannot use our results to prove the analogous statement for $\xi$ smooth. Indeed, if we were able to control all characteristic jets, one would note that these only depend on the $r$-jet of $\xi$ itself. Then, we would conclude by replacing $\xi$ by an analytic distribution with the same $r$-jet.}. It turns out (Proposition \ref{prop:realisableJets}) that this is enough for our goals.

First we look at those jets that lie over the open strata of $\SS_M$. We will study those tangent to smaller strata in the next Subsection.
\begin{definition} \label{def:tangencyJets}
An $r$-jet $\sigma \in J^r(Z_1)$ is \textbf{characteristic of tangency type} if $\sigma \in J^r(S,\Xi_S)$, for some stratum $S \in \SS^o$ not contained in the zero section.

We denote the space of all such jets as $J^r(Z_1,d\lambda)_\SS$.
\end{definition}
Unlike a general characteristic jet, a jet in $J^r(Z_1,d\lambda)_\SS$ extends to a characteristic curve. Indeed, such a jet will be based in the interior of some stratum $S$, so we can extend it to a curve tangent to $\Xi_S$ (using the Ehresmann formalism from Lemma \ref{lem:EhresmannLiftJets}).

The following preliminary Lemma describes the jets of tangency type more concretely:
\begin{lemma} \label{lemma:tangencyJets}
Let $S \in \SS^o$ contained in $Z_n$. Then:
\[ J^r(Z_1,d\lambda)_\SS \cap S = J^r(S) \cap J^r(Z_n,d\lambda) \cap \pi^{-1}(J^r(M,\xi)) = J^r(S) \cap J^r(Z_1,d\lambda), \]
where the intersection on the left-hand side denotes taking the jets of $J^r(Z_1)$ based on $S$ (but not necessarily tangent to it).

In particular, $J^r(Z_1,d\lambda)_\SS \cap S$ is semianalytic.
\end{lemma}
\begin{proof}
For the first equality: According to Definition \ref{def:stratification}, $\Xi_S$ is the collection of vectors tangent to $S$ on which $d\lambda|_{Z_n}$ and the defining forms of $\xi$ vanish, so the claim follows by definition.

For the second equality: We invoke the analogous equality for vectors given in Proposition \ref{prop:liftingHigherZ}.

The last claim follows because the three spaces in the middle  (or the two on the right) are semianalytic. This was shown for $J^r(Z_n,d\lambda)$ in Lemma \ref{lem:characteristicJets} and the proof for $\pi^{-1}(J^r(M,\xi))$ is analogous. Semianalyticity of $J^r(S)$ was justified in Lemma \ref{lem:subsetJets}.
\end{proof}

From which it follows:
\begin{proposition} \label{prop:tangencyJets}
$J^r(Z_1,d\lambda)_\SS$ is a smooth semianalytic subset of $J^r(Z_1)$.
\end{proposition}
\begin{proof}
$J^r(Z_1,d\lambda)_\SS$ is described locally as the graph of an Ehresmann lifting map for jets, which is analytic because each $\Xi_S$ is analytic. This proves smoothness and analyticity at its interior points. 

To prove semianalyticity (including frontier points), we write
\[ J^r(Z_1,d\lambda)_\SS := \coprod_{S \in \SS^o} (J^r(Z_1,d\lambda)_\SS \cap S). \]
The claim then follows from Lemma \ref{lemma:tangencyJets} and the fact that $\SS$ is finite.
\end{proof}

\begin{lemma} \label{lem:tangencyJetDimension}
Let $(M,\xi)$ be analytic and bracket-generating. Then the dimension of $J^r(Z_1,d\lambda)_\SS$ is at most
\[ (\rank(\xi)-1)(r - 1) + 2\dim(M). \]
\end{lemma}
\begin{proof}
It was shown in Lemma \ref{lem:dimensionCountingNR} that $\Xi_S$, for every $S \in \SS^o$, has rank at most $\rank(\xi)-1$. The Ehresmann charts arising from Lemma \ref{lem:EhresmannLiftJets} allow us to explicitly parametrise the piece of $J^r(Z_1,d\lambda)_\SS$ corresponding to $S$ using jets of curves in $\R^{\rank(\Xi_S)}$ (plus a choice of basepoint). Thus:
\begin{align*}
\dim(J^r(Z_1,d\lambda)_\SS \cap S) &\quad\leq\quad \dim(J^r(\R^{\rank(\Xi_S)})) + \dim(S) \\
                                   &\quad =  \quad (\rank(\Xi_S)-1)(r-1) + \dim(S) \\
												 	         &\quad\leq\quad (\rank(\xi)-1)(r - 1) + 2\dim(M).
\end{align*}
\end{proof}

\subsection{Horizontal jets of submanifold type} \label{ssec:submanifoldJets}

Having looked at jets lying over open strata, we now single out those tangent to the smaller strata in $\SS_M$:
\begin{definition} \label{def:submanifoldJets}
A horizontal $r$-jet $\sigma \in J^r(M,\xi)$ is \textbf{of submanifold type} if it is tangent to $S \in \SS_M^c$. In particular, it is tangent to the distribution $\xi \cap TS$.

The space of all jets of submanifold type is denoted by $J^r(M,\xi)_{\SS_M}$.
\end{definition}
Do note that these jets may not be singular at all (but we discard them in case they do contain some singular jets). 

\begin{proposition} \label{prop:submanifoldJets}
$J^r(M,\xi)_{\SS_M}$ is a smooth semianalytic subset of $J^r(Z_1)$. Its dimension is bounded above by:
\[ (\rank(\xi)-1)(r - 1) + \dim(M). \]
\end{proposition}
\begin{proof}
The proof is exactly the same as in Proposition \ref{prop:tangencyJets} and Lemma \ref{lem:tangencyJetDimension}. The equality
\[ J^r(M,\xi)_{\SS_M} = \coprod_{S \in \SS^c_M} J^r(M,\xi) \cap J^r(S), \]
together with Lemma \ref{lem:subsetJets}, proves semianalyticity. Using Ehresmann charts and the dimension counting of Lemma \ref{lem:dimensionCountingNR} proves the claimed dimension bound.
\end{proof}

\subsection{Inadmissible jets} \label{ssec:singularJets}

So far we have only controlled those characteristic/singular jets that are tangent to strata. However, it is possible for a characteristic/ singular curve to jump across strata. Indeed, this is a well-known phenomenon in Subriemannian Geometry, and it is the key obstruction in proving one the main open problems in the field: the \emph{Sard conjecture for the endpoint map} \cite{BFPR,BV,LLMV,LMOPV}. 

To take into account this jumping phenomenon we define:
\begin{definition} \label{def:singularJets}
An $r$-jet $\sigma \in J^r(M,\xi)$ is \textbf{inadmissible} if any of the following conditions holds:
\begin{itemize}
\item It is in the image of the projection $J^r(Z_1,d\lambda)_\SS \to J^r(M,\xi)$.
\item It is contained in $J^r(M,\xi)_{\SS_M}$.
\item It is in the closure of the previous two sets.
\end{itemize}

The space of inadmissible jets is denoted by $J^r(M,\xi)_\sing$.
\end{definition}

From the results in the previous Subsections it readily follows:
\begin{proposition} \label{prop:singularJets}
Let $(M,\xi)$ be analytic and bracket-generating. Then, $J^r(M,\xi)_\sing$ is a closed subanalytic set of $J^r(M,\xi)$ of dimension at most
\[ (\rank(\xi)-1)(r - 1) + 2\dim(M). \]
\end{proposition}
\begin{proof}
According to Proposition \ref{prop:tangencyJets}, the space of tangency type jets $J^r(Z_1,d\lambda)_\SS$ is semianalytic. One may then argue as in Corollary \ref{cor:projectCharacteristicJets} and show that its projection to $J^r(M,\xi)$ is subanalytic (the one subtlety again is taking $(\rho\oplus\eta)$-invariance into account).

Similarly, Proposition \ref{prop:submanifoldJets} shows that the jets of submanifold type $J^r(M,\xi)_{\SS_M}$ form a semianalytic (and, in particular, subanalytic) set.

Additionally, according to Proposition \ref{prop:closureFrontier}, closures of semianalytic sets are also semianalytic, so our previous reasoning applies again. This shows that $J^r(M,\xi)_\sing$ is subanalytic (and closed due to being a closure).

The dimension bounds from Proposition \ref{prop:tangencyJets} and Proposition \ref{prop:submanifoldJets} prove the second claim. We note that taking closures does not increase the dimension.
\end{proof}

The reader may wonder why we are interested in inadmissible jets. Indeed, we could have defined a singular jet to be the image of a characteristic jet. Those are the jets that actually solve the singularity condition up to order $r$, and are thus much more natural to be studied. Here is the fundamental property of inadmissible jets, which allows us to disregard general singular jets:
\begin{proposition} \label{prop:realisableJets}
Let $\gamma: \Op(0) \to (M,\xi)$ be a singular curve. Then, $j^r\gamma \in J^r(M,\xi)_\sing$.
\end{proposition}
\begin{proof}
By local finiteness of the stratification $\SS_M$, there exists a sequence of points $t_i \to_{i\to \infty} 0$, and intervals $\Op(t_i) \supset t_i$, such that the images $\gamma(\Op(t_i))$ are all contained in the same stratum $S \in \SS_M$. Do note that $S$ may not be the stratum containing $\gamma(0)$ (but $\gamma(0) \in \overline{S}$).

Now there are two options: If $S \in \SS_M^c$, the jets $j^r_{t_i}\gamma$ are all of submanifold type. Then $j^r\gamma$ is in the closure of $J^r(M,\xi)_{\SS_M}$. As such, it is inadmissible.

Otherwise, if $S \in \SS_M^0$, we use Hsu's result Proposition \ref{prop:Hsu} to lift $\gamma$ to a characteristic curve $\widetilde\gamma$. As such $j^r\widetilde\gamma \in J^r(Z_1,d\lambda)$. We just need to show that it in fact belongs to $\overline{J^r(Z_1,d\lambda)_\SS}$.

We argue as before: By local finiteness of $\SS$ there is a subsequence $t_i' \to_{i\to \infty} 0$, with neighbourhoods $\Op(t_i') \supset t_i'$, such that the images $\widetilde\gamma(\Op(t_i'))$ are all contained in a stratum $S'$ lying over $S$. The jets $j^r_{t_i'}\widetilde\gamma$ are all characteristic and tangent to $S'$, so they belong to $J^r(Z_1,d\lambda)_\SS$. Their limit $j^r\widetilde\gamma$ then belongs to the closure.
\end{proof}
The point is that a singular jet may be spurious: It may satisfy the singularity condition up to order $r$, but not arise as the jet of an actual singular curve. Such jets are not taken into account in our arguments (but it would be interesting to find an argument to detect/control them as well; we leave this as an open question).

In the same direction: Many of our inadmissible jets are not necessarily singular (for instance, those tangent to lower strata of $\SS_M$), but we discard them for convenience because they form small families nonetheless. It is unclear whether it may be useful to be a bit more careful in our argument and try to refine the class of inadmissible jets further.

\subsection{Microregular jets and the proof of Theorem \ref{thm:main2}} \label{ssec:microregularJets}

We have essentially completed the proof of Theorem \ref{thm:main2}, but let us spell it out.
\begin{definition} \label{def:microregularJets}
A horizontal $r$-jet $\sigma \in J^r(M,\xi)$ is \textbf{microregular} if it is not inadmissible.

The space of all microregular jets is denoted by $J^r(M,\xi)_\microreg$.
\end{definition}

We will sometimes speak of the microregular differential relation $\SR_\microreg$. Do note that it is indeed differential (since it depends on finite jets), but of arbitrarily high order.

\begin{proof}[End of the proof of Theorem \ref{thm:main2}]
According to Proposition \ref{prop:singularJets}, the set $J^r(M,\xi)_\microreg$ is open and dense. Furthermore, the codimension of its complement $J^r(M,\xi)_\sing \subset J^r(M,\xi)$ is bounded below by:
\begin{align*}
\dim(J^r(M,\xi)) - \dim(J^r(M,\xi)_\sing) & \geq [\rank(\xi)(r-1) + \dim(M)] - [(\rank(\xi)-1)(r - 1) + 2\dim(M)] \\
                                          & \geq r - 2\dim(M)- 1,
\end{align*}
which is linear in $r$. This proves the infinite codimension of inadmissible jets as $r$ goes to infinity.

By taking representative germs and invoking Proposition \ref{prop:realisableJets}, it follows that the space of singular germs has infinite codimension within the space of all horizontal germs.

To show that a family of horizontal vectors $(v_k)_{k \in K}$ can be extended to a family of microregular jets, we need an additional observation. Write $J^r(M,\xi,v_k)$ for the horizontal $r$-jets with a given $v_k$ as its first jet. This is an analytic subspace of $J^r(M,\xi)$ of codimension $l$ (which is independent of $r$). As such, the intersection 
\[ J^r(M,\xi)_\sing \cap J^r(M,\xi,v_k) \subset J^r(M,\xi,v_k) \]
is subanalytic and of codimension $O(r)$. By taking $r$ sufficiently large, it has codimension larger than $\dim(K)$, proving that we can find a family of preimages $(\nu_k)_{k \in K}$ in $J^r(M,\xi)_\microreg$. Observe now that the same reasoning applies when our starting point is a family of integral $a$-jets, for any fixed integer $a$.

By taking $r$ to infinity and using the previous reasoning, we deduce that 
\[ J^\infty(M,\xi)_\microreg \to J^1(I,M,\xi) \]
is a Serre fibration. Even further, reasoning as above with $K = \D^N$, for varying $N$ and relative to the boundary, shows weak contractibility of its fibres.
\end{proof}

\begin{proof}[Proof of local integrability in Theorem \ref{thm:main}]
According to Theorem \ref{thm:main2}, any family of vectors $(v_k)_{k \in K}$ in $J^1(M,\xi)$ can be extended, for $r$ sufficiently large, to a family of $r$-jets $(\nu_k)_{k \in K}$ in $J^r(M,\xi)_\microreg$.

In turn, we can choose a $K$-family of germs of horizontal curves $(\gamma_k)_{k \in K}$ representing the jets $(\nu_k)_{k \in K}$. This can be done by working in a Ehresmann chart, taking representative germs in the projection, and lifting them up (Lemma \ref{lem:EhresmannLiftJets}). Being microregular is an open condition for jets (according to Theorem \ref{thm:main2}), so all the jets of $\gamma_k$ (in a sufficiently small interval) are microregular.

The same reasoning applies in the relative case.
\end{proof}

\begin{proof}[Proof of Corollary \ref{cor:immersions}]
First we note that $\SR$ being open means that it defines a subset $R \subset J^a(M)$, for some $a$. Its prolongations (i.e. differential consequences in $J^r(M)$), due to openness, are the preimages of $R$ under the usual forgetful map $J^r(M) \to J^l(M)$.

According to Theorem \ref{thm:main2}, any family of integral $a$-jets $(v_k)_{k \in K}$ in $J^a(M,\xi)$ can be extended, for $r$ sufficiently large, to a family of $r$-jets $(\nu_k)_{k \in K}$ in $J^r(M,\xi)_\microreg$; this holds also relatively. The claim follows by considering families taking values in $R$.
\end{proof}

\section{An instance of Thom transversality} \label{sec:Thom}

In the next Section we will prove the microflexibility (Definition \ref{def:microflexibility}) of the microregular differential relation $\SR_\microreg$, completing the proof of Theorem \ref{thm:main}. The main idea is straightforward: Given any local deformation $(\widetilde\gamma_s)_{s \in [0,1]}$ of a microregular curve $\gamma$, we use regularity in the complement of the deformation region to extend it to a global deformation $(\gamma_s)_{s \in [0,\delta]}$. This is essentially automatic from the definition of regularity. The issue is that the resulting deformation is not necessarily made of microregular curves (although they are regular by construction).

To deal with this, we invoke the fact that the space of inadmissible jets has large codimension (Theorem \ref{thm:main2}) so regularity may be used to avoid it. This can be understood as a form of Thom transversality with respect to the inadmissible jets once we have regularity. In this Section we work out this transversality statement Theorem \ref{thm:Thom} in detail.

\subsection{A preparatory Lemma}

The regularity condition will allow us to deform families of curves while keeping their endpoint fixed. In order to do this, we first need to assign to each regular curve a finite family of variations that realise every possible move of the reduced endpoint.

\begin{lemma} \label{lem:variations}
Let $(W^m,\SD^l)$ bracket generating and not necessarily analytic. Fix a regular horizontal curve $\gamma: [0,1] \to (W,\SD)$. Let $\phi: U \subset \R^m \to W$ be an Ehresmann chart adapted to $\gamma$, with reduced endpoint $\Endpoint_\phi$, and projection $\pi: U \to \R^l$. We abuse notation and regard $\SD$ and $\gamma$ as objects in $U$.

Then, there is a family of curves $(\gamma_v: [0,1] \to (U,\SD))_{v \in \R^{m-l}}$ satisfying:
\begin{itemize}
\item $\gamma_0 = \gamma$.
\item $\gamma_v = \gamma$ in $\Op(\{0\})$.
\item $\pi \circ \gamma_v = \pi \circ \gamma$ in $\Op(\{0,1\})$.
\item The variations $\Gamma_v := \partial_v|_{v=0}(\gamma_v)$ satisfy that $d_{\gamma}\Endpoint_\phi|_{\langle \Gamma_v \rangle_{v \in \R^{m-l}}}$ is surjective.
\end{itemize}

In particular, $\Endpoint_\phi|_{(\gamma_v)_{v \in \Op(0)}}$ is a diffeomorphism with image a little vertical ball:
\[ \Op(\gamma(1)) \subset \{\pi \circ \gamma(1)\} \times \R^{m-l}. \]
\end{lemma}
\begin{proof}
Let $e_{l+1},\cdots,e_m$ be the standard basis of vector fields in $U \subset \R^m$ tangent to the last $(m-l)$-coordinates. By regularity of $\gamma$ there are variational vector fields $\Gamma_i$ that surject onto them using the differential $d_\gamma\Endpoint$ of the restricted endpoint map. We see the $\Gamma_i$ as lifts of compactly-supported variational vector fields $\widetilde\Gamma_i$ of $\pi \circ \gamma$.

We write $\widetilde\Gamma_v := \sum_{i=l+1}^m v_i \widetilde\Gamma_i$ for the (compactly-supported) variational vector field mapping to $v := \sum_{i=l+1}^m v_i e_i$. We arbitrarily extend the $\widetilde\Gamma_v$ to curves $\widetilde\gamma_v$ in $\R^l$ that agree with $\pi \circ \gamma$ close to their endpoints. Their lifts with initial point $\gamma(0)$ yield the claimed family $\gamma_v$.
\end{proof}

\begin{definition}
Under the assumptions and conclusions of the previous Lemma, we say that $(\gamma_v)_{v \in \R^{m-l}}$ is a \textbf{variational endpoint family} associated to $\gamma$ and that $\Endpoint_\phi|_{(\gamma_v)_{v \in \Op(0)}}$ is a \textbf{variational endpoint map}.
\end{definition}

\begin{corollary} \label{cor:variationsParametric}
Let $(U,\SD)$ be a bracket-generating Ehresmann chart. Fix a family of regular horizontal curves $(\gamma_k: [0,1] \to (U,\SD))_{k \in \D^b}$. Then, they admit a variational endpoint family $(\gamma_{k,v})_{k \in \D^b, v \in \R^{m-l}}$, parametrically in $k$.
\end{corollary}
\begin{proof}
To each curve $\gamma_k$ we assign a variational endpoint family $(\gamma_v^k)_{v \in \Op(0)}$ using the Lemma (these do not vary smoothly in $k)$. The corresponding infinitesimal variations $(\Gamma_v^k)_{v \in \R^{m-l}}$ can be extended arbitrarily to the nearby curves; we denote these by $(\Gamma_{k',v}^k)_{k' \in \Op(k), v \in \R^{m-l}}$. Since $d\Endpoint_\phi|_{(\Gamma_v^k)}$ is an isomorphism, the same must be true for $d\Endpoint_\phi|_{(\Gamma_{k',v}^k)}$ by continuity for those $k'$ sufficiently close to $k$. We can thus assume, by a change of basis parametric in $k'$, that $d\Endpoint_\phi(\Gamma_{k',v}^k) = v$.

Now we cover $\D^b$ by finitely many of these opens $\Op(k_i)$. Using a partition of unity $\chi_i$ we set 
\[ \Gamma_{k,v} := \sum_i \chi_i \Gamma_{k,v}^{k_i}. \]
Since $d\Endpoint_\phi$ is linear, it follows that $\Gamma_{k,v}$ maps to $v$. We extend these variational vector fields arbitrarily to horizontal curves (with compactly-supported projection and given initial point).
\end{proof}

\subsection{The proof}

Let us restate Theorem \ref{thm:Thom} in its relative and parametric versions:
\begin{proposition}\label{prop:Thom}
Let $(M,\xi)$ be bracket-generating and real analytic. Let $K' \subset K$ be compact manifolds (possibly with boundary).

Let $(\gamma_k: [0,1] \to (M,\xi))_{k \in K}$ be a family of horizontal curves parametrised by $K$. We assume that:
\begin{itemize}
\item $\gamma_k$ is microregular for all $k \in \Op(K')$.
\item $\gamma_k$ is microregular at every $t \in \Op(\partial([0,1]))$, for all $k$.
\item $\gamma_k$ is regular for all $k$.
\end{itemize}
Then, given any integer $a$, the family can be $C^a$-perturbed, relative to the boundary in the parameter and the domain, to yield a family of microregular horizontal curves.
\end{proposition}

Such a statement was explained already in \cite[Proof of Theorem 1]{PP} and \cite{CP} in the setting of Engel manifolds. The main point (both here and there) is that, unlike most transversality statements, one cannot argue purely locally. Indeed, we need to invoke regularity (which is a global property) in order to produce, during the deformation process, curves with the correct boundary conditions.

\begin{proof}[Proof of Theorem \ref{thm:Thom} and Proposition \ref{prop:Thom}]
We fix a triangulation $\ST$ of $K$. If $\ST$ is sufficiently thin, given any simplex $\sigma \in \ST$, the curves $(\gamma_k)_{k \in \sigma}$ will all be contained in an Ehresmann chart adapted to them.

We have to deform the family $(\gamma_k)_{k \in K}$ to achieve microregularity. This is done one (sufficiently small) neighbourhood $\SU(\sigma)$ of a cell $\sigma \in \ST$ at a time; at every step we invoke Theorem \ref{thm:main2} to perturb the curves in order to avoid the inadmissible jets. We order the $\SU(\sigma)$ arbitrarily but increasingly in dimension. If $K'$ is not empty, we require that $\ST$ extends a triangulation of $K'$; the cells corresponding to $K'$ are ignored in our argument. Standard arguments (see for instance \cite[Proposition 30]{CPPP}) show that:
\begin{itemize}
\item $\SU(\sigma)$ can be identified with $\D^{\dim(K)}$.
\item The neighbourhood of the previous cells is, in this model, of the form $\Op(\NS^{\dim(\sigma)-1})$ (and, in particular, empty if $\sigma$ is a vertex).
\end{itemize}

Now we work in a concrete $\SU := \SU(\sigma) \subset K$. The curves $(\gamma_k)_{k \in \SU}$ are all adapted to an Ehresmann chart $\phi: U \to M$ with projection $\pi: U \to \R^l$. This may be assumed applying Proposition \ref{prop:EhresmannChartCurve} to each $\gamma_k$ and using the fact that $\ST$ can be chosen to be arbitrarily thin (so all curves  $(\gamma_k)_{k \in \SU}$ are contained in the Ehresmann chart associated to one of them). We pass back and forth between $U$ and $M$ to avoid cluttering the notation.

We fix $r = O(n)$ sufficiently large so that $J^r(M,\xi)_\sing$ has codimension larger than $\dim(K)+1+\dim(M)$. We project down $J^r(U,\xi)_\sing$ to $J^r(\R^l)$; the image $J^r(\R^l)_\sing$ is, by construction, closed, subanalytic, and of codimension larger than $\dim(K)+1$. We then apply standard Thom transversality to $(\pi \circ \gamma_k)_{k \in \SU}$ to yield a $C^a$-deformed family $(\widetilde{\pi \circ \gamma}_k)_{k \in \SU}$ whose $r$-jets avoid $J^r(\R^l)_\sing$. In the region corresponding to previous cells we do not have to deform because we had already achieved microregularity there. We also do not need to deform close to the endpoints $\{0,1\}$.

We lift $(\widetilde{\pi \circ \gamma}_k)_{k \in \SU}$ to a family $(\widetilde{\gamma}_k)_{k \in \SU}$ using the Ehresmann lifting map of $U$. The initial points are chosen to be the initial points of $(\gamma_k)_{k \in \SU}$. However, the lifting process does not respect the endpoint: The endpoints of the family family $(\widetilde{\gamma}_k)_{k \in \SU}$ have been displaced in a $C^a$-small manner and therefore they do not satisfy the desired boundary conditions.

To address this we use the results from the previous Subsection. Before we start the induction, we instead apply Corollary \ref{cor:variationsParametric} to $(\gamma_k)_{k \in \SU}$, enlarging it to a family $(\gamma_{k,v})_{k \in \SU, v \in \Op(0)}$. Do note that the $\gamma_{k,0} = \gamma_k$ have the desired endpoint but the others do not. Indeed, the curves $(\gamma_{k,v})_{v \in \Op(0)}$ have the same initial point as $\gamma_k$, but the endpoint varies in a little (vertical) ball.

We may then argue as above but with the family $(\gamma_{k,v})_{k \in \SU; v \in \Op(0)}$ instead; this yields a deformed family $(\widetilde\gamma_{k,v})_{k \in \SU, v \in \Op(0)}$. Since the variational endpoint map of each $\gamma_k$ provided a diffeomorphism between $\Op(v)$ and a little vertical ball, the same is true for the deformed family (here we use the $C^a$-smallness of the perturbation). In particular, for each $k$, there is exactly one curve $\widetilde\gamma_{k,v(k)}$ in $(\widetilde\gamma_{k,v})_{v \in \Op(0)}$ with the desired endpoint. Do note that $\widetilde\gamma_{k,v(k)}$ is smooth, because its projection $\pi \circ \widetilde\gamma_{k,v(k)}$ is smooth.

This concludes the inductive step, producing a family of microregular, horizontal curves $(\widetilde\gamma_{k,v(k)})_{k \in K}$. Now we prove that this family is homotopic to $(\gamma_k)_{k \in K}$ through a regular horizontal family. Indeed, the two are $C^a$-close. In particular, by working in the Ehresmann chart $U$ and lifting, we can interpolate between them using a $C^a$-small family of horizontal and regular curves (but potentially with the wrong endpoint). We apply the closing argument of the previous paragraph to this interpolation. This concludes the proof.
\end{proof}

\begin{remark}
Let us comment on a slightly subtle point in the proof. The set $J^r(\R^l)_\sing$ is a finite collection of immersed submanifolds in $J^r(\R^l)$ (Proposition \ref{prop:subanalyticStrat1}), but not necessarily a stratification satisfying Whitney's conditions (which is usually the assumption in order to invoke Thom transversality with respect to it). However, Whitney's conditions are not needed in our setting, in which the singularity set $J^r(\R^l)_\sing$ is to be avoided altogether due to the high codimension assumption.

Indeed, we cover the submanifolds of $J^r(\R^l)_\sing$ by open balls in $J^r(\R^l)$. The balls covering the frontier of a submanifold cover also a neighbourhood of it, so the collection of balls may be chosen to be finite. Then we argue inductively, starting from the smaller submanifolds. At each step we apply Thom transversality to avoid them.

A similar observation can be found in \cite[Remark in p.33]{Gr86}.
\end{remark}

\section{Microflexibility} \label{sec:microflexibility}

In this last Section we prove the microflexibility of microregular curves. The argument is very similar to the proof of Proposition \ref{prop:Thom}.

\begin{proof}[Proof of microflexibility in Theorem \ref{thm:main}]
Let us recall the setup: We have a family of microregular curves $(\gamma_k: [0,1] \to (M,\xi))_{k \in K}$ and a family of local deformations $(\widetilde\gamma_{k,s}: \Op(I) \to (M,\xi))_{k \in K, s \in [0,1]}$ defined in a neighbourhood of some closed subset $I \subset [0,1]$. We want to extend this deformation to a global one, defined in an arbitrarily small time interval $s \in [0,\delta]$.

First, note that it may be assumed that $K = \D^b$ and that the family $(\gamma_k)_{k \in \D^b}$ maps into an Ehresmann chart $\phi(U)$. This is done as in Proposition \ref{prop:Thom}, by triangulating $K$ and working relatively to previous simplices. Similarly, due to compactness, we can find a finite collection of disjoint intervals $I_i \subset [0,1] \setminus I$ such that, together with $\Op(I)$, they cover $[0,1]$. By working with a particular $I_i$, we reduce the problem to the case in which $I=\{0,1\}$. We proceed under these simplifying assumptions.

Due to microregularity, the curves $(\gamma_k)_{k \in \D^b}$ are regular in both $[0,1/2]$ and $[1/2,1]$. We can apply Corollary \ref{cor:variationsParametric} on each half, keeping the middle fixed, to yield a variational endpoint family for both ends:
\[ (\gamma_{k,v_1,v_2})_{k \in \D^b; v_1,v_2 \in \R^{m-l}}, \] 
i.e. $v_1$ controls the vertical displacement over $\gamma(0)$ and $v_2$ the displacement over $\gamma(1)$. We may assume that, under the local projection $\pi: U \to \R^l$ provided by the Ehresmann chart, the variations are trivial in $\Op(\{0,1/2,1\})$. In particular, for $k$ fixed, all the projected curves agree with $\pi \circ \gamma_k$ at the endpoints and middle point. 

We now modify the variational family close to its endpoints. Namely, we define a new family
\[ (\widetilde\gamma_{k,v_1,v_2,s})_{k \in \D^b; v_1,v_2 \in \R^{m-l}; s \in [0,1]} \]
whose projection $\pi \circ \widetilde\gamma_{k,v_1,v_2,s}$ agrees with $\pi \circ \gamma_{k,v_1,v_2}$ away from the endpoints, but close to them is given by the deformation $\pi \circ \widetilde\gamma_{k,s}$. We choose the middle point as the lifting point so $\widetilde\gamma_{k,v_1,v_2,s}(1/2) = \gamma_k(1/2)$.

Now we are almost done. For $s=0$, the restricted endpoint maps are diffeomorphisms. The same is true from small $s \leq \delta$ by continuity. Therefore, there are unique values $v_1(k,s)$, $v_2(k,s)$ such that the family
\[ (\widetilde\gamma_{k,v_1(k,s),v_2(k,s),s})_{k \in \D^b, s \in [0,\delta]} \]
is a global deformation extending $(\widetilde\gamma_{k,s})_{k \in K, s \in [0,1]}$ for small time. The issue is that this family is regular (because we have explicit variational families showing that this is the case), but not microregular. Then we must apply transversality Theorem \ref{thm:Thom}, relative to the already microregular region, to perturb it slightly and obtain a family satisfying the claimed properties.
\end{proof}

\end{document}